\documentclass[12pt,reqno]{NumPDEsArticle}

\usepackage{fullpage}
\usepackage{hyperref}
\usepackage{amsmath,amssymb,amsthm}
\usepackage{todonotes}
\usepackage{nameref}
\usepackage{biblatex}
\usepackage{booktabs}
\usepackage{bm}
\usepackage{orcidlink}
\usepackage{moreverb}
\usepackage{mathtools, mathrsfs}
\allowdisplaybreaks
\usepackage[shortlabels]{enumitem}
\usepackage{tikz}
\usepackage{tkz-euclide}
\usepackage[caption=false]{subfig}

\usetikzlibrary{shapes.geometric, calc}
\usepackage{xcolor}
\usepackage{xfrac}
\usepackage{filecontents,pgfplots,pgfplotstable}
\usepackage{graphicx}
\usepackage{amsmath}
\usetikzlibrary{matrix}
\definecolor{TUMagenta}{RGB}{186, 70, 130}
\definecolor{TUBlue}{RGB}{0,102,153}
\definecolor{TUGreen}{RGB}{0, 126, 113}
\definecolor{TUYellow}{RGB}{225, 137, 34}
\definecolor{pyBlue}{HTML}{1f77b4}
\definecolor{pyRed}{HTML}{d62728}
\definecolor{pyGreen}{HTML}{2ca02c}
\definecolor{pyOrange}{HTML}{ff7f0e}
\definecolor{pyPurple}{HTML}{9467bd}
\definecolor{pyYellow}{HTML}{bcbd22}
\definecolor{pyGrey}{HTML}{7f7f7f}
\newcommand{\vvvert}{|\mkern-1.5mu|\mkern-1.5mu|}

\DeclareMathOperator*{\essinf}{ess\,inf}
\DeclareMathOperator*{\esssup}{ess\,sup}
\tikzset{reference/.style={thick,dashed}}

\tikzset{reference/.style={thick,dashed}}

\pgfplotsset{%
compat=newest,%
every axis/.style={scale only axis},%
grid style={densely dotted, semithick},%
}

\usepackage[basicdelimiters]{NumPDEsMacros}

\DeclareFieldFormat{titlecase}{#1}

\hypersetup{
colorlinks,
citecolor = ForestGreen,
linkcolor = BrickRed,
urlcolor  = NavyBlue
}

\usepackage[export]{adjustbox}

\title[Cost-optimal AILFEM for locally Lipschitz problems]{Cost-optimal adaptive FEM\\ with linearization and
algebraic solver \\ for semilinear elliptic PDEs}

\author{Maximilian Brunner~\orcidlink{0000-0003-0636-1491}}
\author{Dirk Praetorius~\orcidlink{0000-0002-1977-9830}}
\author{Julian Streitberger~\orcidlink{0000-0003-1189-0611}}
\address{TU Wien, Institute of Analysis and Scientific Computing, Wiedner Hauptstr. 8-10/E101/4, 1040 Vienna, Austria}
\email{maximilian.brunner@asc.tuwien.ac.at \quad \texttt{(corresponding author)}}
\email{dirk.praetorius@asc.tuwien.ac.at}
\email{julian.streitberger@asc.tuwien.ac.at}



\thanks{This research was funded in whole or in part by the Austrian Science Fund (FWF)
  [\href{https://www.fwf.ac.at/en/research-radar/10.55776/F65}{10.55776/F65},
    \href{https://www.fwf.ac.at/en/research-radar/10.55776/I6802}{10.55776/I6802}, and
  \href{https://www.fwf.ac.at/en/research-radar/10.55776/P33216}{10.55776/P33216}].
  For open access purposes, the author has applied a CC BY public copyright license
to any author accepted manuscript version arising from this submission.}

\begin{document}

\maketitle

\begin{abstract}
We consider scalar semilinear elliptic PDEs, where the nonlinearity is strongly monotone, but only locally
Lipschitz continuous. To linearize the arising discrete nonlinear problem, we employ a damped Zarantonello
iteration, which leads to a linear Poisson-type equation that is symmetric and positive definite. The
resulting system is solved by a contractive algebraic solver such as a multigrid method with local smoothing.
We formulate a fully adaptive algorithm that equibalances the various error components coming from mesh
refinement, iterative linearization, and algebraic solver. We prove that the proposed adaptive iteratively
linearized finite element method (AILFEM) guarantees convergence with optimal complexity, where the rates are
understood with
respect to the overall computational cost (i.e., the computational time). Numerical experiments investigate
the involved adaptivity parameters.
\end{abstract}


\section{Introduction}
\subsection{Problem setting and main results}
Undoubtedly, adaptive finite element methods (AFEMs) are in the canon of reliable numerical methods for the
solution of partial differential equations (PDEs). Some of the seminal contributions in this still very
active area are~\cite{bv1984, doerfler1996, mns2000, bdd2004, stevenson2007, ckns2008, ks2011, cn2012,
ffp2014} for linear problems,~\cite{veeser2002, dk2008, bdk2012, gmz2012, ghps2021} for nonlinear problems,
and~\cite{axioms} for an abstract framework.

By means of conforming finite elements, this paper is concerned with the cost-optimal computation of the
solution $u^\star \in H_0^1(\Omega)$ to the \emph{semilinear} elliptic model problem
\begin{equation}\label{eq:modelproblem}
  -\operatorname{div}(\boldsymbol{A} \nabla u^{\star}) + b(u^{\star}) = F \quad \text{ in } \Omega \quad
  \text{ subject to } \quad
  u^{\star}=0 \quad \text{ on } \partial \Omega,
\end{equation}
with a Lipschitz domain $\Omega \subset \mathbb{R}^d$ for $d \in \{1,2,3\}$, an elliptic diffusion coefficient
$\boldsymbol{A}\colon \Omega \to \mathbb{R}_{\textup{sym}}^{d\times d}$, a monotone nonlinearity $b \colon
\Omega \to \mathbb{R}$, and
sufficiently regular data $F$. The assumptions are such that the Browder--Minty theorem ensures existence and
uniqueness.

Moreover, the model problem~\eqref{eq:modelproblem} can be recast into the framework of strongly monotone and
locally Lipschitz continuous operators such that the abstract model problem reads: For $\mathcal{X} = H_0^1(\Omega)$
with topological dual space $\mathcal{X}' = H^{-1}(\Omega)$ and duality bracket $\langle \, \cdot \, , \,
\cdot \, \rangle$, a nonlinear
operator $\mathcal{A}\colon \mathcal{X} \to \mathcal{X}'$, and given data $F\in \mathcal{X}'$, we aim to
approximate the solution $u^\star \in \mathcal{X}$ to
\begin{equation}\label{eq:weakform}
  \langle \, \mathcal{A} u^\star \, , \, v \, \rangle = \langle \, F \, , \, v \, \rangle \quad \text{ for
  all } v \in \mathcal{X}.
\end{equation}
To this end, we employ conforming piecewise polynomial finite element spaces $\mathcal{X}_H \subset \mathcal{X}$ with
the corresponding discrete solution $u^\star_H \in \mathcal{X}_H$ to
\begin{equation}\label{eq:weakform:discrete}
  \langle \, \mathcal{A} u^\star_H \, , \, v_H \, \rangle = \langle \, F \, , \, v_H \, \rangle \quad \text{
  for all } v_{H} \in \mathcal{X}_H,
\end{equation}
which, however, can hardly be computed exactly, since~\eqref{eq:weakform:discrete} is still a discrete
nonlinear system of equations.

The major difficulty of such problems is that the Lipschitz constant of $\mathcal{A}$ depends on the considered
functions $v$ and $w$ in the sense that for $\vartheta> 0$, it holds that
\begin{equation}
  \| \mathcal{A} v -\mathcal{A} w \|_{\mathcal{X}'} \le  L[\vartheta]\, \vvvert v- w \vvvert \quad  \text{
  for all } v, w \in \mathcal{X}
  \text{ with } \max\big\{ \vvvert v \vvvert, \vvvert w \vvvert \big\} \le \vartheta. \tag{LIP$'$}
\end{equation}
Moreover, this dependence also appears in the stability constant of the residual-based \textsl{a~posteriori}
error estimator~\cite{v2013, bbimp2022}.

Hence, for such a problem class, any approximate numerical scheme must ensure uniform boundedness of all
computed approximations $u_H^\star \approx u_H \in \mathcal{X}_H$ throughout the algorithm. This
constitutes the first main result: The developed \emph{adaptive iteratively linearized FEM}  (AILFEM)
algorithm (more detailed in Algorithm~\ref{algorithm:AILFEM} below) guarantees a uniform upper bound on all
iterates (see Theorem~\ref{theorem:uniformBoundedness} below). In particular, the algorithm steers the
decision whether it is more preferable to refine the mesh adaptively or to do an additional step of
linearization or a further algebraic solver step instead.

Once uniform boundedness is established, we prove full R-linear convergence
(Theorem~\ref{theorem:RLinearConvergence} below) as the second main result. Full R-linear convergence
establishes contraction in each step of the algorithm regardless of the algorithmic decision. At the expense
of a more challenging analysis that links energy arguments with the energy norm of the algebraic solver, full
R-linear convergence is guaranteed for all mesh levels $\ell \ge \ell_0=0$ while prior
works~\cite{bbimp2022, bhimps2023} used compactness arguments which only guaranteed the existence of the index
$\ell_0\in \mathbb{N}_0$ (and not necessarily $\ell_0=0$). As a consequence of uniform boundedness and full R-linear
convergence, the third main result proves optimal rates both understood with respect to the degrees of
freedom and with respect to the overall computational cost (Corollary~\ref{cor:rates=complexity} and
Theorem~\ref{th:optimal_complexity}) of the proposed algorithm.

Compared to existing results in the literature~\cite{ghps2021, hpsv2021, hpw2021, fps2023},
all three main results require a suitable adaptation of the stopping criteria of the linearization loop as
well as sufficiently many iterations in the algebra loop, together with subtle technical challenges, in
particular, for the proof of full R-linear convergence.

\subsection{From AFEM to AILFEM}
On each mesh level (with mesh index $\ell$), the arising discrete nonlinear problems cannot be solved exactly
in practice as supposed in classical AFEM~\cite{veeser2002, dk2008, bdk2012, gmz2012}. To deal with this
issue, we follow~\cite{cw2017, ghps2018, hw2020:ailfem} and consider the so-called \emph{Zarantonello
iteration} from~\cite{zarantonello1960} as a linearization method (with index $k$). The Zarantonello
iteration is a Richardson-type iteration where only a Laplace-type problem has to be solved in each
iteration. Since the arising large SPD systems are still expensive to solve exactly, we employ a contractive
algebraic solver as a nested loop to solve the Zarantonello system inexactly (with iteration index $i$). The
loops thus come with a natural nestedness (see Figure~\ref{fig:solveAndEstimate}), where the overall
schematic loop of the algorithm reads
\medskip
\begin{center}
  \includegraphics[width = 0.7\textwidth]{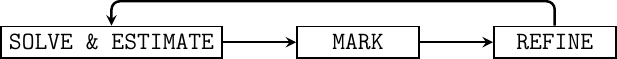}
\end{center}
\medskip
\begin{figure}
  \begin{center}
    \includegraphics[width = \textwidth]{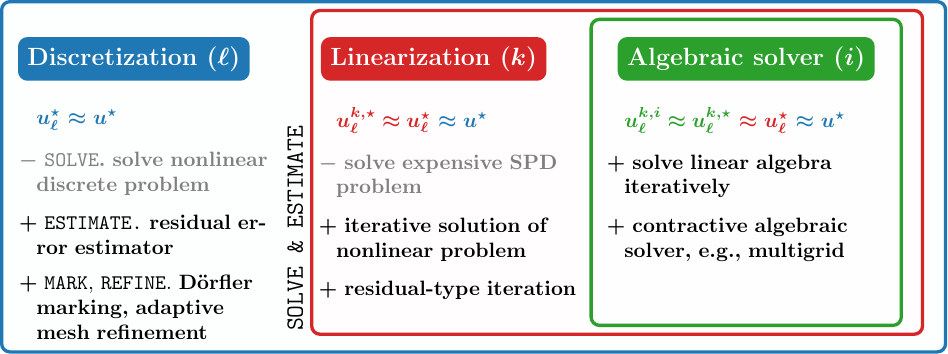}
  \end{center}
  \caption{Depiction of the nested loops of the AILFEM algorithm~\ref{algorithm:AILFEM} below.}
  \label{fig:solveAndEstimate}
\end{figure}
Since the proposed adaptive loop depends on all previous computations, optimal convergence rates should
be understood with respect to the overall computational cost. This idea of \emph{optimal complexity}
originates from the wavelet community~\cite{cdd2001, cdd2003} and was last used in the context of AFEM
in~\cite{stevenson2007} for the Poisson model problem and~\cite{cg2012} for the Poisson eigenvalue problem,
both under realistic assumptions on generic iterative solvers.

AILFEMs with iterative and/or inexact solver with \textsl{a~posteriori} error estimators are found in,
e.g.,~\cite{bms2010, aev2011, agl2013, ev2013, aw2015, cw2017}  and references therein. Besides the
Zarantonello iteration, for globally Lipschitz continuous nonlinearities, the works~\cite{hw2020:convergence,
hw2020:ailfem, hpw2021} analyze also other linearizations such as the Ka\v{c}anov iteration or damped Newton
schemes. Optimal complexity of the Zarantonello loop that is coupled with an algebraic loop is analyzed
in~\cite{bhimps2023} for nonsymmetric second-order linear elliptic PDEs and for strongly monotone (and
globally Lipschitz continuous) model problems in~\cite{ghps2018, ghps2021, hpsv2021,  hpw2021, fps2023}.

The literature on AILFEMs for locally Lipschitz continuous problems is scarce and closing this gap is the aim
of this work. The semilinear model problem is treated in, e.g.,~\cite{aw2015} by a damped Newton iteration
and in~\cite{ahw2022} by an energy-based approach with experimentally observed optimal rates. We also refer
to the own work~\cite{bbimp2022cost} for an AILFEM with optimal rates with respect to the overall
computational cost, however, under the assumption that the arising linear systems can be solved
at linear cost.
More precisely, compared to the previous work~\cite{bbimp2022cost}, in this paper we also take an optimal
  algebraic solver for the linearized problem into account and propose an adaptive algorithm
  ensuring optimal convergence
  rates with respect to the computation time. Moreover, compared to~\cite{bbimp2022cost} that elaborates the
  proof of full R-linear convergence
  along the lines of~\cite{ghps2021}, we provide a much simpler proof inspired by~\cite{fps2023}.
  However, the work~\cite{fps2023}
  employs a general quasi-orthogonality from~\cite{feischl2022} that is not available for nonlinear problems
  in general since
  the proof relies on a stable LU-decomposition of the linear problem. Therefore, to avoid compactness
  arguments like in~\cite{bhimps2023}, we employ orthogonality in the underlying energy.

\subsection{Outline}
This paper is structured as follows: Section~\ref{section:monlip} introduces the abstract framework on
locally Lipschitz continuous operators. In Section~\ref{section:algorithm}, we formulate the (idealized)
AILFEM algorithm (Algorithm~\ref{algorithm:AILFEM}). We prove uniform boundedness for the final iterates of
the algebraic solver (Theorem~\ref{theorem:uniformBoundedness}).  Section~\ref{sec:Rlinear} presents the
second main result: Full R-linear convergence (Theorem~\ref{theorem:RLinearConvergence}). In particular,
rates with respect to the degrees of freedom coincide with rates with respect to the computational cost
(Corollary~\ref{cor:rates=complexity}). In Section~\ref{section:optimality}, we prove the main result on
optimal complexity of the proposed AILFEM algorithm (Theorem~\ref{th:optimal_complexity}). In
Section~\ref{section:numerics}, we present numerical experiments of the proposed AILFEM strategy and
investigate its optimal complexity for various choices of the adaptivity parameters.

\section{Strongly monotone operators}
\label{section:monlip}

This section introduces an abstract framework of strongly monotone and locally Lipschitz continuous
operators. This class of operators covers the model problem~\eqref{eq:modelproblem} of semilinear elliptic
PDEs with monotone semilinearity.

\subsection{Abstract model problem} \label{subsection:modelproblem}
Let $\mathcal{X}$ be a Hilbert space over $\mathbb{R}$ with scalar product $\langle\!\langle \cdot \, , \,
\cdot \rangle\!\rangle$ and induced norm $\vvvert \cdot \vvvert$. Let $\mathcal{X}_H \subseteq \mathcal{X}$
be a closed subspace. Let $\mathcal{X}'$ be the dual space with norm $\| \cdot \|_{\mathcal{X}'}$ and denote
by $\langle \cdot \, , \, \cdot \rangle$ the duality bracket on $\mathcal{X}' \times \mathcal{X}$. Let
$\mathcal{A} \colon \mathcal{X} \to \mathcal{X}'$ be a nonlinear operator. We suppose that $\mathcal{A}$ is
{\bf strongly monotone}, i.e., there exists a monotonicity constant $\alpha > 0$ such that
\begin{equation}\label{eq:strongly-monotone}
  \alpha \, \enorm{v - w}^2 \le \langle \mathcal{A} v - \mathcal{A} w \, , \, v - w \rangle
  \quad \text{ for all } v, w \in \mathcal{X}. \tag{SM}
\end{equation}
Moreover, we suppose that $\mathcal{A}$ is {\bf locally Lipschitz continuous}, i.e., for all $\vartheta > 0$,
there exists $L[\vartheta] >0$ such that
\begin{equation}\label{eq:locally-lipschitz}
  \langle \mathcal{A} v - \mathcal{A} w \, , \, \varphi \rangle \le L[\vartheta] \, \enorm{v - w} \, \enorm{\varphi}
  \quad \text{ for all } v, w, \varphi \in \mathcal{X}
  \text{ with } \max\big\{ \enorm{v}, \enorm{v - w}\big\} \le \vartheta. \tag{LIP}
\end{equation}

\begin{remark}
  We remark that local Lipschitz continuity is often defined differently in the existing literature,
  cf.~\cite[p.~565]{zeidler}: For all $\Theta > 0$, there exists $L'[\Theta] >0$ such that
  \begin{equation}\label{eq:locally-lipschitz:variant}\tag{LIP$'$}
    \langle \mathcal{A} v - \mathcal{A} w \, , \, \varphi \rangle \le L'[\Theta] \, \enorm{v - w} \, \enorm{\varphi}
    \text{ for all } v, w, \varphi \in \mathcal{X}
    \text{ with } \max\big\{ \enorm{v}, \enorm{w}\big\} \! \le \! \Theta.
  \end{equation}
  We note that the conditions~\eqref{eq:locally-lipschitz} and~\eqref{eq:locally-lipschitz:variant} are
  indeed equivalent in the sense that~\eqref{eq:locally-lipschitz}
  yields~\eqref{eq:locally-lipschitz:variant} with $\Theta = 2 \vartheta$, and,
  conversely,~\eqref{eq:locally-lipschitz:variant} yields~\eqref{eq:locally-lipschitz} with $\vartheta = 2 \, \Theta$.
  However, condition~\eqref{eq:locally-lipschitz} is better suited for the inductive proof of
  Proposition~\ref{prop:zarantonello} which is the main ingredient to guarantee uniform boundedness in
  Theorem~\ref{theorem:uniformBoundedness}.
\end{remark}
Without loss of generality, we may suppose that $\mathcal{A}0 \neq F \in \mathcal{X}'$. We consider the
operator equation:
Seek $u^\star \in \mathcal{X}$ that solves~\eqref{eq:weakform}.
For any closed subspace $\mathcal{X}_H \subseteq \mathcal{X}$, we consider the corresponding Galerkin
discretization~\eqref{eq:weakform:discrete}.
We note existence and uniqueness of the solutions to~\eqref{eq:weakform}--\eqref{eq:weakform:discrete} and a
C\'{e}a-type estimate.

\begin{proposition}[{\cite[Proposition~2]{bbimp2022cost}}]\label{prop:existence}
  Suppose that $\mathcal{A}$ satisfies~\eqref{eq:strongly-monotone} and~\eqref{eq:locally-lipschitz}.
  Then,~\eqref{eq:weakform}--\eqref{eq:weakform:discrete} admit unique solutions $u^\star \in \mathcal{X}$ and
  $u_H^\star \in \mathcal{X}_H$, respectively, and
  \begin{equation}\label{eq:exact:bounded}
    \max\big\{ \enorm{u^\star}, \enorm{u^\star_H} \big\} \le M \coloneqq \frac{1}{\alpha} \, \| F
    - \mathcal{A} 0 \|_{\mathcal{X}'} > 0
  \end{equation}
  as well as
  \begin{equation}\label{eq:cea}
    \enorm{u^\star - u_H^\star} \le C_{\textup{Céa}} \min_{v_H \in \mathcal{X}_H}  \enorm{u^\star
    -v_H} \quad \text{ with } \quad C_{\textup{Céa}} = L[2M]/\alpha.\hfill \qed
  \end{equation}
\end{proposition}

Finally, we suppose that $\mathcal{A}$ has a potential $\mathcal{P}$: There exists a G\^{a}teaux
differentiable function $\mathcal{P}
\colon \mathcal{X} \to \mathbb{R}$ such that its derivative $\textup{d}{\mathcal{P}}\colon \mathcal{X} \to
\mathcal{X}'$ coincides with $\mathcal{A}$, i.e.,
\begin{equation}\label{eq:potential}
  \dual{\mathcal{A} w}{v} = \dual{\textup{d}{\mathcal{P}} (w)}{v} = \lim_{\substack{t \to 0 \\ t \in
  \mathbb{R}}}\, \frac{\mathcal{P}(w+tv) -
  \mathcal{P}(w)}{t} \quad \text{ for all } v, w\in \mathcal{X}. \tag{POT}
\end{equation}
With the energy $\mathcal{E} (v) \coloneqq (\mathcal{P}-F) v$, there holds the following classical equivalence.
\begin{lemma}[\phantom{}{see, e.g.,~\cite[Lemma~5.1]{ghps2018}}]\label{lemma:equivalence}
  Let $\mathcal{X}_H \subseteq \mathcal{X}$ be a closed subspace (where also $\mathcal{X}_H$ replaced by
  $\mathcal{X}$ is admissible).
  Suppose that $\mathcal{A}$ satisfies~\eqref{eq:strongly-monotone}, \eqref{eq:locally-lipschitz},
  and~\eqref{eq:potential}. Let $\vartheta \ge M$. Let $v_{H} \in \mathcal{X}_{H}$ with $\enorm{
  v_H - u^\star_{H}} \le \vartheta$. Then, it holds that
  \begin{equation}\label{eq:equivalence}
    \frac{\alpha}{2} \, \enorm{v_{H} - u^\star_{H}}^2 \le  \mathcal{E}(v_{H}) -
    \mathcal{E}(u^\star_{H}) \le  \frac{L[\vartheta]}{2}\, \enorm{v_{H} - u^\star_{H}}^2.
  \end{equation}
  In particular, the solution $u^\star_H$ of the variational formulation~\eqref{eq:weakform:discrete} is indeed the unique
  minimizer of $\mathcal{E}$ in $\mathcal{X}_H$, i.e.,
  \begin{equation}\label{eq:emin}
    \mathcal{E}(u^\star_H) \le \mathcal{E}(v_H) \quad \text{ for all } v_H \in \mathcal{X}_H.
  \end{equation}
  In particular, it holds that
  \begin{equation}\label{eq:qo:abstract}
    \mathcal{E}(v_{H}) - \mathcal{E}(u^\star) =   \big[\mathcal{E}(v_{H}) - \mathcal{E}(u^\star_{H})\big]
    +\big[\mathcal{E}(u^\star_{H}) - \mathcal{E}(u^\star)\big] \quad \text{ for all } v_H \in \mathcal{X}_H
  \end{equation}
  and all these energy differences are nonnegative. \hfill \qed
\end{lemma}
\subsection{Iterative linearization and algebraic solver}\label{subsection:innerloops}
Let $\mathcal{X}_H \subset \mathcal{X}$ be a finite-dimen\-sional (and hence closed) subspace of $\mathcal{X}$.
In order to solve the arising nonlinear discrete problems~\eqref{eq:weakform:discrete}, we will incorporate a
linearization method as well as an algebraic solver into the proposed algorithm.
\paragraph{\quad \bf Linearization by Zarantonello iteration.}
For a detailed discussion of the Zarantonello iteration, we refer to~\cite[Section~2.2--2.4]{bbimp2022cost}.
For a damping parameter \(\delta > 0\) and $w_H \in \mathcal{X}_H$, let \(\Phi_{H}(\delta;
w_H) \in \mathcal{X}_H\) solve
\begin{equation}\label{eq:Zarantonello:iteration}
  \edual{\Phi_H(\delta; w_H)}{v_H} = \edual{w_H}{v_H} + \delta \,
  \bigl[F(v_H) - \dual{\mathcal{A} w_H}{v_H}\bigr] \quad \text{ for all } v_H \in \mathcal{X}_H.
\end{equation}
The Lax--Milgram lemma proves existence and uniqueness of $\Phi_H(\delta; w_H)$, i.e., the
Zarantonello operator $\Phi_{H}(\delta; \cdot) \colon \mathcal{X}_H \to \mathcal{X}_H$ is well-defined. In
particular, $u_{H}^\star = \Phi(\delta; u_{H}^\star)$ is the unique fixed point of
$\Phi_{H}(\delta; \cdot)$ for any damping parameter $\delta > 0$.  Moreover, for sufficiently small
\(\delta>0\), the Zarantonello operator is norm-contractive.

\begin{proposition}[{see, e.g.,~\cite[Proposition~4]{bbimp2022cost}}]\label{prop:zarantonello}
  Suppose that $\mathcal{A}$ satisfies~\eqref{eq:strongly-monotone} and~\eqref{eq:locally-lipschitz}. Let $\vartheta
  > 0$ and $v_H, w_H \in \mathcal{X}_H$ with $\max \big\{ \enorm{v_H},
  \enorm{v_H-w_H} \big\} \le \vartheta$. Then, for all $0 < \delta < 2\alpha / L[\vartheta]^2$
  and $0 < q_{\textup{Zar}}^{\star}[\delta, \vartheta]^2 \coloneqq 1 - \delta\,(2\alpha - \delta
  L[\vartheta]^2) < 1$, it holds that
  \begin{equation}\label{eq:ZarantonelloExact}
    \enorm{\Phi_H(\delta; v_H) - \Phi_H(\delta; w_H)}
    \le q_{\textup{Zar}}^{\star}[\delta, \vartheta] \, \enorm{v_H - w_H}.
  \end{equation}
  We note that $q_{\textup{Zar}}^{\star}[\delta, \vartheta] \to 1$ as $\delta \to 0$. For known $\alpha$ and
  $L[\vartheta]$, the contraction constant $q_{\textup{Zar}}^{\star}[\delta, \vartheta]^2 =
  1-\alpha^2/L[\vartheta]^2 = 1-
  \alpha \, \delta $ is minimal and only attained for $\delta = \alpha/L[\vartheta]^2$. \qed
\end{proposition}
\paragraph{\quad \bf Algebraic solver.}
The Zarantonello system~\eqref{eq:Zarantonello:iteration} leads to an SPD system of equations to compute
$\Phi_H(\delta; u_H)$. Since large SPD problems are still computationally expensive, we employ an
iterative algebraic solver with process function \(\Psi_H \colon \mathcal{X}' \times \mathcal{X}_H \to
\mathcal{X}_H\) to solve the arising system~\eqref{eq:Zarantonello:iteration}. More precisely, given a linear
functional \(\varphi \in \mathcal{X}'\) and an approximation \(w_H \in \mathcal{X}_H\) of the exact solutions
\(w_H^\star \in \mathcal{X}_H\) to
\begin{equation*}
  \edual{w_H^\star}{v_H} = \varphi(v_H)
  \quad \text{ for all } v_H \in \mathcal{X}_H,
\end{equation*}
the algebraic solver returns an improved approximation \(\Psi_H(\varphi; w_H) \in \mathcal{X}_H\)
in the sense that there exists a uniform constant \(0 < q_{\textup{alg}} < 1\) independent of \(\varphi\) and
\(\mathcal{X}_H\) such that
\begin{equation}\label{eq:algebraContraction}
  \enorm{w_H^\star - \Psi_H(\varphi; w_H)}
  \le
  q_{\textup{alg}} \, \enorm{w_H^\star - w_H} \quad \text{ for all } w_H \in \mathcal{X}_H.
\end{equation}
To simplify notation when the right-hand side $\varphi$ is complicated or lengthy (as for the Zarantonello
iteration~\eqref{eq:Zarantonello:iteration}), we shall write $\Psi_H(w_H^\star; \cdot)$ instead
of $\Psi_H(\varphi; \cdot)$, even though $w_H^\star$ is unknown and will never be computed.

\subsection{Mesh refinement}
Henceforth, let \(\mathcal{T}_0\) be an initial triangulation of \(\Omega\) into compact triangles. For mesh
refinement, we use newest vertex bisection (NVB); cf.~\cite{stevenson2008} for $d \ge 2$ with admissible
$\mathcal{T}_0$ as well as~\cite{kpp2013} for \(d = 2\) and~\cite{dgs2023} for $d\ge 2$ with nonadmissible
$\mathcal{T}_0$.
For $d=1$, we refer to~\cite{affkp2013}. For each triangulation \(\mathcal{T}_H\) and marked elements
\(\mathcal{M}_H \subseteq \mathcal{T}_H\), let \(\mathcal{T}_h \coloneqq \texttt{refine}(\mathcal{T}_H, \mathcal{M}_H)\)
be the coarsest refinement of \(\mathcal{T}_H\) such that at least all elements \(T \in \mathcal{M}_H\) have been
refined, i.e., \(\mathcal{M}_H  \subseteq \mathcal{T}_H \setminus \mathcal{T}_h\). We write \(\mathcal{T}_h \in
\mathbb{T}(\mathcal{T}_H)\) if \(\mathcal{T}_h\) can be obtained from \(\mathcal{T}_H\) by finitely many
steps of NVB, and,
for \(N \in \mathbb{N}_0\), we write \(\mathcal{T}_h \in \mathbb{T}_N(\mathcal{T}_H)\) if \(\mathcal{T}_h \in
\mathbb{T}(\mathcal{T}_H)\) and
\(\# \mathcal{T}_h - \# \mathcal{T}_H \le N\). To abbreviate notation, let \(\mathbb{T} \coloneqq
\mathbb{T}(\mathcal{T}_0)\). Throughout,
any \(\mathcal{T}_H \in \mathbb{T}\) is associated with a finite-dimensional space \(\mathcal{X}_H \subset
\mathcal{X}\) such
that nestedness of meshes \(\mathcal{T}_h \in \mathbb{T}(\mathcal{T}_H)\) implies nestedness of the associated spaces
\(\mathcal{X}_H \subseteq \mathcal{X}_h\).

\subsection{Axioms of adaptivity and \textsl{a~posteriori} error estimator}\label{subsection:axioms}
For $\mathcal{T}_H \in \mathbb{T}$, $T \in \mathcal{T}_H$, and $v_H \in \mathcal{X}_H$, let~$\eta_H(T,
v_H) \in \mathbb{R}_{\ge 0}$ be the local contributions of an \textsl{a~posteriori} error estimator and abbreviate
\begin{equation}\label{eq:estimator:generic}
  \eta_H(v_H) \coloneqq \eta_H(\mathcal{T}_H, v_H), \text{ where }
  \eta_H(\mathcal{U}_H, v_H)
  \coloneqq
  \Big(\! \! \sum_{T \in \,\mathcal{U}_H} \!\! \eta_H(T, v_H)^2 \Big)^{1/2}  \text{for all }
  \mathcal{U}_H \subseteq \mathcal{T}_H.
\end{equation}
We suppose that the error estimator $\eta_H$ satisfies the following axioms of adaptivity
from~\cite{axioms} with a slightly relaxed variant of stability~\eqref{axiom:stability} in the spirit
of~\cite{bbimp2022}.

\begin{enumerate}[label=\textbf{\textrm{(A\arabic*)}}, ref=A\arabic*, leftmargin=1cm]
  \item\label{axiom:stability} \textbf{stability:}
    For all $\vartheta > 0$ and all $\mathcal{U}_H \subseteq \mathcal{T}_h \cap \mathcal{T}_H$, there exists
    $C_{\textup{stab}}[\vartheta] >0$ such that for all $v_h \in \mathcal{X}_h$ and $v_H \in \mathcal{X}_H$ with
    $\max\big\{\enorm{v_h}, \enorm{v_h-v_H}\big\}  \le \vartheta$, it holds that
    \begin{equation*}
      \big| \eta_h(\mathcal{U}_H, v_h) - \eta_H(\mathcal{U}_H, v_H) \big|
      \le C_{\textup{stab}}[\vartheta]\, \enorm{v_h - v_H}.
    \end{equation*}
  \item\label{axiom:reduction} \textbf{reduction:} \rm
    With $0 < q_{\textup{red}} <1$, it holds that
    \begin{equation*}
      \eta_h(\mathcal{T}_h \backslash \mathcal{T}_H, v_H)
      \le q_{\textup{red}} \, \eta_H(\mathcal{T}_H \backslash \mathcal{T}_h, v_H) \quad \text{ for all }
      v_H \in \mathcal{X}_H.
    \end{equation*}
    \bf
  \item\label{axiom:reliability} \textbf{reliability:} \rm
    There exists $C_{\textup{rel}} >0$ such that
    \begin{equation*}
      \enorm{u^\star - u_H^\star}
      \le C_{\textup{rel}} \, \eta_H(u_H^\star).
    \end{equation*}
    \bf
  \item\label{axiom:discreteReliability} \textbf{discrete reliability:} \rm
    There exists $C_{\textup{drel}} >0$ such that
    \begin{equation*}
      \enorm{u_h^{\star} - u_H^{\star}}
      \le C_{\textup{drel}} \, \eta_H(\mathcal{T}_H \backslash \mathcal{T}_h, u_H^{\star})  .
    \end{equation*}
\end{enumerate}

\subsection{\texorpdfstring{Application of abstract framework~(\ref{eq:weakform}) to semilinear
PDEs~(\ref{eq:modelproblem})}{Application of abstract framework~(2) to semilinear PDEs~(1)}}
In the following, we comment on how the semilinear PDE~\eqref{eq:modelproblem} fits into the abstract
framework in Section~\ref{subsection:modelproblem}--\ref{subsection:axioms}. Let $\Omega \subset \mathbb{R}^d$, $d\in
\{ 1,2,3 \}$, be a bounded Lipschitz domain with polygonal boundary.
The weak formulation of the semilinear model problem~\eqref{eq:modelproblem} reads: Given $F \in
H^{-1}(\Omega)$, find $u^\star \in \mathcal{X} \coloneqq H_0^1(\Omega)$ such that
\begin{equation}\label{eq:modelproblem:weak}
  \dual{\boldsymbol{A} \, \nabla u^\star}{\nabla v}_{\Omega} + \dual{b(u^\star)}{v}_{\Omega} = \dual{F}{v}
  \quad \text{ for all } v \in H_0^1(\Omega),
\end{equation}
where $\dual{\cdot}{\cdot}_{\Omega}$ denotes the $L^2(\Omega)$-scalar product. Note
that~\eqref{eq:modelproblem:weak} coincides with~\eqref{eq:weakform}, where $\mathcal{A} u \coloneqq
\dual{\boldsymbol{A} \, \nabla u}{\nabla \,\cdot}_{\Omega} + \dual{b(u)}{\cdot}_{\Omega}$ with $u \in
\mathcal{X}$. As a means of discretization, we consider Lagrange finite element spaces of piecewise
  polynomial functions of a fixed polynomial degree $p \in \mathbb{N}_0$ on a conforming triangulation
  $\mathcal{T}_H$ of $\Omega$, namely $\mathcal{X}_H \coloneqq \mathbb{S}_0^p(\mathcal{T}_H) \coloneqq \{ v_H
    \in H_0^1(\Omega) \colon v_H|_T \text{ is polynomial of degree } \le p \text{ for all } T \in \mathcal{T}_H
  \}$. This discretization leads to nested spaces $\mathcal{X}_H \subseteq \mathcal{X}_h$ whenever $\mathcal{T}_H \in \mathbb{T}$ and $\mathcal{T}_h
\in \mathbb{T}(\mathcal{T}_H)$.
The precise assumptions on the model problem are given as follows.

\paragraph{\quad \bf Assumptions on the right-hand side.}
We suppose the following.
\begin{enumerate}[label = \textrm{(RHS)}, ref=RHS, leftmargin=1.1cm]
  \item Let $\dual{F}{v} \coloneqq \dual{f}{v}_{\Omega} + \dual{\boldsymbol{f}}{\nabla v}_{\Omega}$ with given
    $f \in L^{2}(\Omega)$ and $\boldsymbol{f} \in [L^{2}(\Omega)]^{d}$. \label{assump:rhs}
\end{enumerate} 

\paragraph{\quad \bf Assumptions on the diffusion coefficient.}
The diffusion coefficient $\boldsymbol{A}$ satisfies the following standard assumptions:
\begin{enumerate}[label = \textrm{(ELL)}, ref=ELL, leftmargin=1.1cm]
  \item $\boldsymbol{A} \in L^\infty( \Omega; \mathbb{R}_{\rm sym}^{d\times d})$, where $\boldsymbol{A}(x)$ is a
    symmetric and uniformly positive definite matrix, i.e., the minimal and maximal eigenvalues satisfy
    \begin{equation*}
      0 < \mu_0 \coloneqq \essinf \limits_{x \in \Omega} \, \lambda_{\rm min} (\boldsymbol{A}(x))
      \le \esssup_{x \in \Omega} \, \lambda_{\rm max} (\boldsymbol{A}(x)) \eqqcolon \mu_1 < \infty.
    \end{equation*}
    \label{assump:ell}
\end{enumerate}\vspace{-\baselineskip}
In particular, the $\boldsymbol{A}$-induced energy scalar product $\edual{v}{w} \coloneqq
\dual{\boldsymbol{A}\nabla v}{\nabla w}_{\Omega}$ induces an equivalent norm $\enorm{v} \coloneqq
\edual{v}{v}^{1/2}$ on $H^1_0(\Omega)$.

\paragraph{\quad \bf Assumptions on the nonlinear reaction coefficient.}
The nonlinearity $b(\cdot)$ satisfies the following assumptions from~\cite[(A1)--(A3)]{bhsz2011}:
\begin{enumerate}[label = \textrm{(CAR)}, ref=CAR, leftmargin=1.2cm]
  \item $b\colon \Omega \times \mathbb{R} \to \mathbb{R}$ is a {\bf \textit{Carathéodory} function}, i.e.,
    for all $n \in
    \mathbb{N}_0$, the $n$-th derivative  $b^{(n)} \coloneqq\partial_{\xi}^n b$ of $b$ with respect to the second
    argument $\xi$ satisfies that
    \begin{itemize}
      \item[$\triangleright$] for any $\xi \in \mathbb{R}$, the function $x \mapsto b^{(n)}(x,\xi)$ is
        measurable on $\Omega$,
      \item[$\triangleright$] for any $x \in \Omega$, the function $\xi \mapsto b^{(n)}(x,\xi)$ exists and is
        continuous in $\xi$.
    \end{itemize} \label{assump:car}
\end{enumerate}
\begin{enumerate}[label = \textrm{(MON)}, ref=MON, leftmargin=1.3cm]
  \item We assume \textbf{monotonicity} in the second argument, i.e., $b'(x, \xi)\coloneqq b^{(1)}(x, \xi)
    \ge 0$ for all $x \in \Omega$ and $\xi \in \mathbb{R}$. By considering $\tilde{b}(v)\coloneqq b(v) - b(0)$ and
    $\tilde{f} \coloneqq f - b(0)$, we assume without loss of generality that $b(x,0)=0$. \label{assump:mon}
\end{enumerate}
To establish continuity of $ v \mapsto \dual{b(v)}{w}_{\Omega}$, we impose the following \textbf{growth
condition} on $b(v)$; see, e.g., \cite[Chapter III,~(12)]{fk1980} or~\cite[(A4)]{bhsz2011}:
\begin{enumerate}[label = \textrm{(GC)}, ref=GC, leftmargin=1cm]
  \item There exist $R > 0$ and $N \in \mathbb{N}$ with $N \le 5$ for $d=3$ such that
    \begin{align*}
      |b^{(N)}(x,\xi)| \le R \quad \text{ for a.e. } x \in \Omega \text{ and all } \xi \in \mathbb{R}.
    \end{align*} \label{assump:poly}
\end{enumerate}\vspace{-\baselineskip}
These assumptions suffice to prove that the operator $\mathcal{A} \coloneqq \mathcal{X} \to \mathcal{X}' =
H^{-1}(\Omega)$ associated
with the model problem~\eqref{eq:modelproblem:weak} is strongly monotone~\eqref{eq:strongly-monotone} and
locally Lipschitz continuous~\eqref{eq:locally-lipschitz} in the sense of
Section~\ref{subsection:modelproblem}; see~\cite[Lemma~20]{bbimp2022cost}.

\paragraph{\quad \bf Energy minimization.} Associated with the semilinear model
problem~\eqref{eq:modelproblem:weak}, we consider the \textbf{energy}
\begin{equation*}\label{eq:semilinear:energy}
  \mathcal{E}(v)\! =\! \frac{1}{2} \int_\Omega \!  |\boldsymbol{A}^{1/2}\nabla v|^2  \, \textup{d}{x} +\! \int_\Omega
  \int_0^{v(x)}\! b(s)  \, \textup{d}{s} \, \textup{d}{x} -\! \int_\Omega \! f  v \, \textup{d}{x}-\!  \int_\Omega
  \!\! \boldsymbol{f} \cdot \nabla v
  \, \textup{d}{x} \, \, \text{ for } v \in H_0^1(\Omega).
\end{equation*}
To ensure the well-posedness of integrals,
we require the following stronger growth condition (guaranteeing \textbf{compactness} of the nonlinear
reaction term). Indeed, the same assumption is also required for stability~\eqref{axiom:stability} of the
residual error estimator~\eqref{eq:estimator:primal} below.
\begin{enumerate}[label=\textrm{(CGC)}, ref=CGC, leftmargin=1.25cm]
  \item There holds~\eqref{assump:poly}, if $d \in \{1,2\}$. If $d=3$, there holds~\eqref{assump:poly} with
    the stronger assumption $ N \in \{ 2, 3\}$. \label{assump:compact}
\end{enumerate}

\paragraph{\quad \bf Residual error estimator.}
To guarantee well-posedness, we additionally require that $\boldsymbol{A}|_T \in [W^{1, \infty}(T)]^{d \times
d}$ and $\boldsymbol{f}|_T \in [W^{1, \infty}(T)]^{d}$ for all $T \in \mathcal{T}_0$, where $\mathcal{T}_0$
is the initial
triangulation of the adaptive algorithm. Then, for $\mathcal{T}_H \in \mathbb{T}$ and $v_H \in \mathcal{X}_H$, the
local contributions of the standard residual error estimator~\eqref{eq:estimator:generic} for the semilinear
model problem~\eqref{eq:modelproblem:weak} read
\begin{align}
  \begin{split}\label{eq:estimator:primal}
    \eta_H(T, v_H)^2 &\coloneqq h_T^2 \,\| f + \operatorname{div}(\boldsymbol{A} \, \nabla v_H - \boldsymbol{f}) -
    b(v_H) \|_{L^2(T)}^2  \\
    & \quad + h_T \, \| \lbrack\!\lbrack (\boldsymbol{A} \, \nabla v_H - \boldsymbol{f} ) \, \cdot \,
    \boldsymbol{n} \rbrack\! \rbrack
    \|_{L^2(\partial T
    \cap \Omega)}^2,
  \end{split}
\end{align}
where $h_T = |T|^{1/d}$ and where $\lbrack\!\lbrack \,\cdot\, \rbrack\!\rbrack$ denotes the jump across edges
(for $d=2$) resp.\ faces
(for $d=3$) and $\boldsymbol{n}$ denotes the outer unit normal vector. For $d=1$, these jumps vanish, i.e.,
$\lbrack\!\lbrack \,\cdot\, \rbrack\!\rbrack = 0$. The \emph{axioms of adaptivity} are established for the
present setting in~\cite{bbimp2022}.

\begin{proposition}[\phantom{}{\cite[Proposition~15]{bbimp2022}}]
  Suppose \eqref{assump:rhs}, \eqref{assump:ell}, \eqref{assump:car}, \eqref{assump:mon},
  and \eqref{assump:compact}. Suppose that NVB is employed as a refinement strategy. Then, the residual
  error estimator from~\eqref{eq:estimator:primal}
  satisfies~\eqref{axiom:stability}--\eqref{axiom:discreteReliability} from Section~\ref{subsection:axioms}.
  The constant $C_{\textup{rel}}$ depends only on $d$,  $\mu_0$, and uniform shape regularity of the initial mesh
  $\mathcal{T}_0$. The constant $C_{\textup{drel}}$ depends, in addition, on the polynomial degree $p$, and
  $C_{\textup{stab}}[\vartheta]$
  depends furthermore on $|\Omega|$, $\vartheta$, $N$, $R$, and $\boldsymbol{A}$. \hfill \qed
\end{proposition}

\paragraph{\quad \bf Algebraic solver.}
As an algebraic solver, we employ a norm-contractive solver to solve the Zarantonello
system~\eqref{eq:Zarantonello:iteration}. Possible choices are, e.g., an optimally preconditioned conjugate
gradient method~\cite{cnx2012} or an optimal geometric multigrid~\cite{wz2017, imps2022}. More precisely, the
numerical experiments below employ the $hp$-robust multigrid method from~\cite{imps2022}, which is
well-defined owing to ellipticity~\eqref{assump:ell}.


\section{Fully adaptive algorithm} \label{section:algorithm}

In this section, we present the adaptive iterative linearized finite element method (AILFEM). As a first main
result, we prove that the iterates from the proposed algorithm are uniformly bounded.

\subsection{Fully adaptive algorithm}

In this section, we introduce a fully adaptive algorithm that steers mesh refinement ($\ell$),  linearization
($k$) and the algebraic solver ($i$). The algorithm utilizes specific stopping indices denoted by an
underline, namely \(\underline{\ell},\underline{k}[\ell], \underline{i}[\ell,k]\). However, we may omit the
dependence when it is apparent
from the context, such as in the abbreviation \(u_\ell^{k, \underline{i}}\coloneqq u_\ell^{k,\underline{i}[\ell,k]}\).

\begin{algorithm}[adaptive iterative linearized FEM (AILFEM)]
  \label{algorithm:AILFEM}
  \phantom{2}\\  {\bfseries Input:} Initial mesh $ \mathcal{T}_0$, marking parameters $0 < \theta \le 1$, $C_{\rm
  mark} \ge 1$, solver parameters $\lambda_{\textup{lin}},\lambda_{\textup{alg}} > 0$, minimal number of
  algebraic solver steps $i_{\textup{min}}
  \in \mathbb{N}$, initial guess $u_0^{0,0}  \coloneqq u_0^{0,\star}\coloneqq u_{0}^{0, \underline{i}} \in
  \mathcal{X}_0$ with
  $\enorm{u_0^{0,0}} \le 2 M$, and Zarantonello damping parameter \(\delta > 0\).

  \noindent
  {\bfseries Adaptive loop:} For all $\ell = 0, 1, 2, \dots$, repeat the following steps
  \eqref{alg:solve+estimate}--\eqref{alg:refine}:
  \begin{enumerate}[label = (\Roman*), ref=\Roman*, font=\upshape]
    \item\label{alg:solve+estimate} {\tt SOLVE \& ESTIMATE.} For all $k = 1, 2, 3, \dots$, repeat steps
      \eqref{alg:phi}--\eqref{alg:terminate_k}:
      \begin{enumerate}[label = (\alph*), ref=\alph*, font=\upshape]
        \item\label{alg:phi} Define $u_\ell^{k,0} \coloneqq u_\ell^{k-1, \underline{i}}$ and, for only
          theoretical reasons,
          $u_\ell^{k,\star} \coloneqq \Phi_\ell(\delta; u_\ell^{k-1,\underline{i}})$.
        \item For all $i=1,2,3, \dots$ repeat steps \eqref{alg:solve}--\eqref{alg:terminate_i}:
          \begin{enumerate}[label = (\roman*), ref=\roman*, font=\upshape]
            \item\label{alg:solve} Compute $u_\ell^{k, i} \coloneqq \Psi_\ell(u_\ell^{k,\star};
              u_\ell^{k, i-1})$ and error estimator $\eta_\ell(u_\ell^{k,i})$.
            \item\label{alg:terminate_i} Terminate the $i$-loop and define $\underline{i}[\ell, k] \coloneqq i$ if
              \begin{equation}\label{eq:i_stopping_criterion}
                \enorm{u_\ell^{k,i-1} - u_\ell^{k, i}} \le \lambda_{\textup{alg}} \, \bigl[\lambda_{\textup{lin}}
                \,\eta_{\ell}(u_\ell^{k,i})  + \enorm{u_\ell^{k, i} - u_\ell^{k,0}} \bigr]  \quad
                \mathtt{AND} \quad i_{\textup{min}} \le i.
              \end{equation}
          \end{enumerate}
        \item\label{alg:terminate_k} Terminate the $k$-loop and define \(\underline{k}[\ell] \coloneqq k\) if
          \begin{equation}\label{eq:k_stopping_criterion}
            \mathcal{E}(u_\ell^{k, 0})- \mathcal{E}(u_\ell^{k, \underline{i}}) \le \lambda_{\textup{lin}}^2
            \, \eta_\ell(u_\ell^{k, \underline{i}})^2
            \quad  \mathtt{ AND } \quad  \enorm{u_\ell^{k, \underline{i}}} \le 2M.
          \end{equation}
      \end{enumerate}

    \item  \texttt{MARK.} Find a set $\mathcal{M}_\ell \in \mathbb{M}_\ell[\theta,
      u_\ell^{\underline{k},\underline{i}}] \coloneqq
      \{\mathcal{U}_\ell \! \subseteq \! \mathcal{T}_\ell \colon \theta \, \eta_\ell(u_\ell^{\underline{k},
        \underline{i}})^2 \le
      \eta_\ell(\mathcal{U}_\ell, u_\ell^{\underline{k}, \underline{i}})^2\}$ such that
      \begin{equation}\label{eq:doerfler}
        \# \mathcal{M}_\ell \le C_{\textup{mark}}  \min_{\mathcal{U}_\ell \in \mathbb{M}_\ell[\theta,
            u_\ell^{\underline{k},
        \underline{i}}]} \# \mathcal{U}_\ell.
      \end{equation}

    \item\label{alg:refine}  \texttt{REFINE.} Generate the new mesh $\mathcal{T}_{\ell+1} \coloneqq  {\tt refine}
      (\mathcal{M}_\ell, \mathcal{T}_\ell)$ by
      employing NVB and define $u_{\ell+1}^{0,0} \coloneqq u_{\ell+1}^{0,\underline{i}} \coloneqq u_{\ell+1}^{0,
      \star} \coloneqq u_\ell^{\underline{k},\underline{i}}$ (nested iteration).
  \end{enumerate}

  \noindent
  {\bfseries Output:}  Sequences of successively refined triangulations $\mathcal{T}_\ell$, discrete approximations
  $u_\ell^{k, i}$ and corresponding error estimators $\eta_\ell(u_\ell^{k, i})$.
\end{algorithm}

For the analysis of Algorithm~\ref{algorithm:AILFEM}, we define the countably infinite index set
\begin{equation*}
  \mathcal{Q} \coloneqq \{(\ell, k, i) \in \mathbb{N}_0^3 \colon u_\ell^{k, i} \text{
  is used in Algorithm~\ref{algorithm:AILFEM}}\},
\end{equation*}
where, for any $(\ell, 0,0)\in \mathcal{Q}$, the final indices are defined as
\begin{alignat*}{2}
  \underline{\ell} &\coloneqq \sup \{\ell \in \mathbb{N}_0 \colon (\ell, 0, 0) \in \mathcal{Q}\} &&\in
  \mathbb{N}_0 \cup \{ \infty\},
  \\
  \underline{k}[\ell] &\coloneqq \sup \{k \in \mathbb{N} \colon (\ell, k, 0) \in
  \mathcal{Q}\}&&\in \mathbb{N} \cup \{ \infty\},
  \\
  \underline{i}[\ell, k] &\coloneqq \sup \{i \in \mathbb{N} \colon (\ell, k, i) \in
  \mathcal{Q}\}&&\in \mathbb{N} \cup \{ \infty\}.
\end{alignat*}
We note, first, that these definitions are consistent with those of Algorithm~\ref{algorithm:AILFEM}, second,
that Lemma~\ref{lemma:termination-i} below proves that $\underline{i}[\ell, k] < \infty$, and, third, that hence
either \(\underline{\ell} =
\infty\) or $\underline{\ell} < \infty$ with
\(\underline{k}[\underline{\ell}] = \infty\). For all $(\ell, k, i) \in \mathcal{Q}$, we introduce the total
step counter $\vert \cdot, \cdot,
\cdot \vert$ defined by
\begin{equation*}
  |\ell, k, i| \coloneqq \#
  \{(\ell^\prime, k^\prime, i^\prime) \in \mathcal{Q} \colon \! (\ell^\prime, k^\prime, i^\prime) < (\ell,k, i) \} 
  = \sum_{\ell' = 0}^{\ell-1} \sum_{k' = 1}^{\underline{k}[\ell']}
  \sum_{i'=1}^{\underline{i}[\ell', k']} 1
  + \sum_{k' = 1}^{k-1} \sum_{i'=1}^{\underline{i}[\ell, k']} 1
  + \sum_{i'=1}^{i-1} 1.
\end{equation*}
We note that this definition provides a lexicographic ordering on \(\mathcal{Q}\).

In the later application to AILFEM for semilinear elliptic PDEs, every step of
Algorithm~\ref{algorithm:AILFEM} can be performed
in linear
complexity as the following arguments show.
\begin{itemize}
  \item[$\triangleright$] \texttt{SOLVE.} The employed algebraic solver is an $hp$-robust
    multigrid~\cite{imps2022} and hence each algebraic solver step requires only $\mathcal{O}(\#
    \mathcal{T}_\ell)$ operations.
  \item[$\triangleright$] \texttt{ESTIMATE.} The simultaneous computation of the standard error indicators
    $\eta_\ell(T, u_\ell^{k,i})$ for all $T \in \mathcal{T}_\ell$ can be done at the cost of $\mathcal{O}(\#
    \mathcal{T}_\ell)$.
  \item[$\triangleright$] \texttt{MARK.} The employed Dörfler marking (and the involved determination of
    $\mathcal{M}_\ell$) is indeed a linear complexity problem; see~\cite{stevenson2007} for $C_{\textup{mark}} = 2$
    and~\cite{pp2020} for $C_{\textup{mark}} = 1$.
  \item[$\triangleright$] \texttt{REFINE.} The refinement of $\mathcal{T}_\ell$ is based on NVB and, owing to the
    mesh-closure estimate~\cite{bdd2004, stevenson2008}, requires only linear cost $\mathcal{O}(\# \mathcal{T}_\ell)$.
\end{itemize}
Thus, the total work until and including the computation of $u_\ell^{k,i}$ is proportional to
\begin{equation}\label{eq:cost}
  \mathtt{cost}(\ell, k, i)
  \coloneqq\! \! \! \! \! \!\sum_{\substack{(\ell', k', i') \in \mathcal{Q} \\ |\ell', k', i'|\le |\ell, k,
  i|}}\! \! \! \!\! \! \# \mathcal{T}_{\ell'} =
  \sum_{\ell^\prime = 0}^{\ell-1} \sum_{k' = 1}^{\underline{k}[\ell^\prime]} \sum_{i'
  = 1}^{\underline{i}[\ell', k']} \# \mathcal{T}_{\ell^\prime}
  + \sum_{k' = 1}^{k-1} \sum_{i'=1}^{\underline{i}[\ell, k']} \# \mathcal{T}_{\ell}
  + \sum_{i'=1}^{i} \# \mathcal{T}_{\ell}.
\end{equation}

An important observation is that the algebraic solver loop always terminates.

\begin{lemma}\label{lemma:termination-i}
  Independently of the adaptivity parameters $\theta$, $\lambda_{\textup{lin}}$, and
  $\lambda_{\textup{alg}}$, the $i$-loop of
  Algorithm~\ref{algorithm:AILFEM} always terminates, i.e., $\underline{i}[\ell,k] < \infty$ for all
  $(\ell,k,0) \in \mathcal{Q}$.
\end{lemma}

\begin{proof}
  We argue as in~\cite[Lemma~3.2]{bhimps2023}.
  Let $(\ell,k,0) \in \mathcal{Q}$.
  We argue by contradiction and assume that the $i$-loop stopping criterion~\eqref{eq:i_stopping_criterion}
  in Algorithm~\ref{algorithm:AILFEM}{\rm(I.b.ii)} always fails and hence $\underline{i}[\ell,k] = \infty$. By
  assumption~\eqref{eq:algebraContraction}, the algebraic solver $\Psi_\ell(u_\ell^{k,\star}; \, \cdot \,)$
  is contractive and hence convergent with limit $u_\ell^{k,\star} \coloneqq \Phi_\ell(\delta;
  u_\ell^{k-1,\underline{i}})$ from Algorithm~\ref{algorithm:AILFEM}{\rm(I.a)}. Moreover, by failure of the stopping
  criterion~\eqref{eq:i_stopping_criterion} in Algorithm~\ref{algorithm:AILFEM}{\rm (I.b.ii)}, we thus obtain that
  \begin{equation*}
    \eta_{\ell}(u_\ell^{k,i}) + \enorm{u_\ell^{k,i} - u_\ell^{k, 0}}
    \eqreff*{eq:i_stopping_criterion}\lesssim \enorm{u_\ell^{k, i} - u_\ell^{k,i-1}}
    \xrightarrow{i \to \infty} 0.
  \end{equation*}
  This yields $\enorm{u_\ell^{k,\star} - u_\ell^{k,0}} = 0$ and hence $u_\ell^{k,\star} = u_\ell^{k,i}$
  for all $i \in \mathbb{N}_0$, since the algebraic solver is contractive. Consequently, the $i$-loop stopping
  criterion~\eqref{eq:i_stopping_criterion} in Algorithm~\ref{algorithm:AILFEM}{\rm (I.b.ii)} will be
  satisfied for $i = i_{\textup{min}}$. This contradicts our assumption, and hence we conclude that
  $\underline{i}[\ell,k] < \infty$.
\end{proof}
\subsection{Energy contraction for the inexact Zarantonello iteration}
In this section, we prove uniform boundedness of the iterates $u_\ell^{k, i}$ from
Algorithm~\ref{algorithm:AILFEM}: Note that the algorithm does not compute the Zarantonello iterate
$u_\ell^{k,\star} \coloneqq \Phi_\ell(\delta; u_\ell^{k-1,\underline{i}})$ exactly, but relies on an approximation
$u_\ell^{k,\underline{i}} \approx u_\ell^{k,\star}$. We prove that this inexact Zarantonello iteration is
contractive with respect to the energy, which is the case if at least $i_{\textup{min}} \in \mathbb{N}$ steps
of the contractive
algebraic solver are performed, i.e., $\underline{i}[\ell, k] \ge i_{\textup{min}}$. In particular, a
suitable choice of the
damping parameter $\delta > 0$ and the index $i_{\textup{min}}$ are derived in the following.

\begin{theorem}\label{theorem:uniformBoundedness}
  Suppose that $\mathcal{A}$ satisfies \eqref{eq:strongly-monotone}, \eqref{eq:locally-lipschitz},
  and \eqref{eq:potential}. With $M$ from~\eqref{eq:exact:bounded},
  define $\tau \coloneqq M + 3M \big(\frac{L[3M]}{\alpha}\big)^{1/2} \ge 4M$.
  Let $\lambda_{\textup{lin}}, \lambda_{\textup{alg}} >0$ and $0 < \theta \le 1$ be arbitrary. Suppose that
  $\enorm{u_\ell^{0,0}} =
  \enorm{u_\ell^{0, \underline{i}}}\le 2M$ with $M>0$ from~\eqref{eq:exact:bounded}. Choose
  $i_{\textup{min}} \in \mathbb{N}$ such that
  \begin{equation}\label{eq:imin}
    q_{\textup{alg}}^{i_{\textup{min}}} \le 1/3.
  \end{equation}
  Then, for any choice of $\delta > 0$ satisfying $0 < \delta < \min\{\frac{1}{L[5\tau]}, \frac{2\alpha}{L[2
  \tau]^2}\}$, there exists a uniform energy contraction constant $0 < q_{\mathcal{E}} =
  q_{\mathcal{E}}[\delta, \tau]<1$
  (see~\eqref{eq:energycontr:const} below) such that the following holds.
  \begin{alignat}{5}
    &\triangleright\, \text{\rm{\bf nested iteration:}}& \enorm{u_\ell^{\underline{k}, \underline{i}}}  & \le
    2M  \quad&&
    \text{ if } (\ell, \underline{k}, \underline{i}) \in \mathcal{Q};\label{crucial:nestediteration}\\
    & \triangleright\, \text{\rm{\bf $\underline{i}$-uniform bound:}} & \enorm{u_\ell^{k,\underline{i}}} &\le
    \tau \quad &&
    \text{ if } (\ell, k, \underline{i}) \in \mathcal{Q}; \label{crucial:mbounded} \\
    & \triangleright\, \text{\rm{\bf $\mathcal{E}$-contraction:}} &  \mathcal{E}(u_\ell^{k+1, \underline{i}})
    - \mathcal{E}(u_\ell^\star)
    &\le q_{\mathcal{E}}^2 \, \big( \mathcal{E}(u_\ell^{k, \underline{i}}) - \mathcal{E}(u_\ell^\star)\big) &&\text{ if
    } (\ell, k+1, \underline{i}) \in
    \mathcal{Q}. \label{crucial:energyContraction}
    \intertext{With~\eqref{crucial:nestediteration}--\eqref{crucial:energyContraction}, we obtain for all iterates the}
    &\triangleright\, \text{\rm{\bf uniform bound:}} & \enorm{u_\ell^{k, i}} &\le 5 \tau \quad && \text{ if }
    (\ell, k, i) \in \mathcal{Q}. \label{eq:uniform:all}
  \end{alignat}
  Moreover, there exists an index $k_0=k_0[\delta, \tau, \alpha, L[3M], M] \in \mathbb{N}$ independently of the mesh
  refinement index $\ell$ such that, for all $k'\ge k_0$, the nested iteration condition $\enorm{u_\ell^{k',
  \underline{i}}} \le 2M$ in the $k$-loop stopping criterion~\eqref{eq:k_stopping_criterion} is always met.
\end{theorem}
The main observation of the following lemma is that the uniform boundedness is passed on by the inexact
Zarantonello iteration along the $k$-loop indices.
\begin{lemma}
  \label{lemma:f4}
  Suppose that $\mathcal{A}$ satisfies \eqref{eq:strongly-monotone}, \eqref{eq:locally-lipschitz},
  and \eqref{eq:potential}. Let $\lambda_{\textup{lin}}, \lambda_{\textup{alg}} > 0$ be arbitrary and define
  $\tau \coloneqq M + 3M
  \big(\frac{L[3M]}{\alpha}\big)^{1/2} \ge 4M$. Let $k \in \mathbb{N}_0$ with $0 \le k < \underline{k}[\ell]$ and
  \begin{equation}\label{eq:f4:assumption}
    \enorm{u_{\ell}^{k, \underline{i}}} \le \tau.
  \end{equation}
  Then, for $i_{\textup{min}} \in \mathbb{N}$ satisfying~\eqref{eq:imin} and for any $0 < \delta <
  \min\{\frac{1}{L[5\tau]},
  \frac{2\alpha}{L[2 \tau]^2}\}$, it holds that
  \begin{align}
    \begin{split}\label{eq:f4}
      0 \le \Big( \frac{1}{2 \delta} - \frac{L[5\tau]}{2}\Big)\,& \enorm{u^{k+1,\underline{i}}_{{\ell}} -
      u^{k,\underline{i}}_{{\ell}}}^2
      \le  \mathcal{E}(u^{k,\underline{i}}_{{\ell}}) - \mathcal{E}(u^{k+1,\underline{i}}_{{\ell}}) \\
      &\le
      \Big( \frac{1}{\delta \,(1-q_{\textup{alg}}^{i_{\textup{min}}})} - \frac{\alpha}{2}\Big)\,
      \enorm{u^{k+1,\underline{i}}_{{\ell}} -
      u^{k,\underline{i}}_{{\ell}}}^2 \text{ for all } (\ell, k+1, \underline{i}) \in \mathcal{Q}.
    \end{split}
  \end{align}
\end{lemma}

\begin{proof}
  The proof is subdivided into five steps.

  \textbf{Step 1 (choice of $i_{\min}$).} We note that for any $i_{\textup{min}}\in \mathbb{N}$, the
  property~\eqref{eq:imin} is
  indeed equivalent to
  \begin{equation}\label{eq:choice:imin}
    \frac{1}{2} \stackrel{!}\le \frac{1-2 \, q_{\textup{alg}}^{i}}{1-q_{\textup{alg}}^{i}}\quad \text{ for
    all } i \ge i_{\textup{min}}.
  \end{equation}

  \textbf{Step 2 (boundedness).}
  Define $e^{k+1}_{{\ell}} \coloneqq u^{k+1,\underline{i}}_{{\ell}} - u^{k,\underline{i}}_{{\ell}}$. Recall
  that for $0 <
  \delta <2\alpha / L[2\tau]^2$, the Zarantonello iteration satisfies
  contraction~\eqref{eq:ZarantonelloExact}. Hence, the contraction of the algebraic
  solver~\eqref{eq:algebraContraction}, the triangle inequality, nested iteration $u_{\ell}^{k+1,0} =
  u_{\ell}^{k,\underline{i}}$, assumption~\eqref{eq:f4:assumption}, and $4M \le \tau$ show that
  \begin{align*}
    \enorm{e^{k+1}_{{\ell}}}
    &\le   \enorm{u_{\ell}^{k+1, \star}\!-\! u_{\ell}^{k+1,\underline{i}}} + \enorm{u_{\ell}^{k+1, \star} \!-\!
    u_{\ell}^{k, \underline{i}}}
    \\
    &\eqreff*{eq:algebraContraction}\le q_{\textup{alg}}^{\underline{i}[\ell,
    k+1]} \,\enorm{u_{\ell}^{k+1,
    \star} \!-\! u_{\ell}^{k+1,0}} + \enorm{u_{\ell}^{k+1, \star} \!-\! u_{\ell}^{k, \underline{i}}} \notag\\
    &\le 2 \,\enorm{u_{\ell}^{k+1, \star} - u_{\ell}^{k, \underline{i}}}
    \le 2 \,\big[\enorm{u_{\ell}^\star - u_{\ell}^{k,\underline{i}}} + \enorm{u_{\ell}^\star-
    u_{\ell}^{k+1, \star}}\big] \\
    & \eqreff*{eq:ZarantonelloExact}\le 2 \,(1+q_{\textup{Zar}}^{\star}[\delta, 2\tau]) \,\enorm{u_{\ell}^\star-
    u_{\ell}^{k, \underline{i}}} \eqreff{eq:f4:assumption}\le 4(M+\tau) \le 5 \tau. \notag
  \end{align*}
  With the convexity of the norm and $\enorm{u_\ell^{k,\underline{i}}} \le \tau \le 5 \tau$, we also obtain that
  \begin{equation}\label{eq:critical}
    \enorm{u_\ell^{k+1, \underline{i}}}  \le \max_{0 \le t \le 1} \, \enorm{u^{k,\underline{i}}_{{\ell}} - t\,
    e^{k+1}_{{\ell}}} \le 5\tau.
  \end{equation}

  \textbf{Step 3.}
  Since the energy $\mathcal{E} = \mathcal{P} - F$ from~\eqref{eq:potential} is G\^{a}teaux differentiable, it
  follows that
  $\varphi(t) \coloneqq \mathcal{E}(u^{k,\underline{i}}_{{\ell}} + t\, e^{k+1}_{{\ell}})$ is differentiable with
  \begin{equation}\label{eq:phiPrime}
    \varphi^\prime(t)  = \dual{\mathrm{d} \mathcal{E}(u^{k,\underline{i}}_{{\ell}} + t\,
    e^{k+1}_{{\ell}})}{e^{k+1}_{{\ell}}} =
    \dual{\mathcal{A}(u^{k,\underline{i}}_{{\ell}} + t\, e^{k+1}_{{\ell}}) - F}{e^{k+1}_{{\ell}}}.
  \end{equation}
  The fundamental theorem of calculus and the exact Zarantonello iteration~\eqref{eq:Zarantonello:iteration} show that
  \begin{align}
    &\mathcal{E}(u^{k,\underline{i}}_{{\ell}}) - \mathcal{E}(u^{k+1,\underline{i}}_{{\ell}}) =  \varphi(0) -
    \varphi(1) \notag
    \\
    &= - \int_0^1
    \varphi^\prime(t) \, \textup{d}t \eqreff{eq:phiPrime}=  - \int_0^1
    \dual{\mathcal{A}(u^{k,\underline{i}}_{{\ell}} + t\,
    e^{k+1}_{{\ell}})-F}{e^{k+1}_{{\ell}}} \, \textup{d}t \notag\\
    & = - \int_0^1 \dual{\mathcal{A}(u^{k,\underline{i}}_{{\ell}} + t\, e^{k+1}_{{\ell}})-
    \mathcal{A}(u^{k,\underline{i}}_{{\ell}})}{e^{k+1}_{{\ell}}} \, \textup{d}t - \dual{
    \mathcal{A}(u^{k,\underline{i}}_{{\ell}})-F}{e^{k+1}_{{\ell}}}  \notag\\
    & \stackrel{\mathclap{\eqref{eq:Zarantonello:iteration}}}{=}  - \int_0^1
    \dual{\mathcal{A}(u^{k,\underline{i}}_{{\ell}} +
    t \,e^{k+1}_{{\ell}})- \mathcal{A} (u^{k,\underline{i}}_{{\ell}})}{e^{k+1}_{{\ell}}} \, \textup{d}t +
    \frac{1}{\delta}\,
    \edual{ u_{\ell}^{k+1, \star} -u_{\ell}^{k,\underline{i}}}{e^{k+1}_{{\ell}}} \label{eq:f4:aux:upper}.
  \end{align}

  \textbf{Step 4 (proof of lower bound in~(\ref{eq:f4})).} For any $i \in \mathbb{N}$ with $i \le
  \underline{i}[\ell,k]$, the
  contraction~\eqref{eq:algebraContraction} of the algebraic solver and nested iteration
  $u_{\ell}^{k, \underline{i}} = u_{\ell}^{k+1, 0}$ prove that
  \begin{align*}
    \enorm{u_{\ell}^{k+1, \star} - u_{\ell}^{k+1, \underline{i}}}
    &\eqreff*{eq:algebraContraction}\le
    q_{\textup{alg}}^{\underline{i}[\ell, k+1]} \,\enorm{u_{\ell}^{k+1, \star} - u_{\ell}^{k, \underline{i}}}
    \\
    &\le
    q_{\textup{alg}}^{i}\, \enorm{u_{\ell}^{k+1, \star} - u_{\ell}^{k+1, \underline{i}}} +
    q_{\textup{alg}}^{i}\, \enorm{u_{\ell}^{k+1,
    \underline{i}} - u_{\ell}^{k, \underline{i}}}.
  \end{align*}
  This gives rise to the \textsl{a~posteriori} estimate
  \begin{equation}\label{eq:aposteriori:alg}
    \enorm{u_{\ell}^{k+1,\star}-u_{\ell}^{k+1,\underline{i}}} \le
    \frac{q_{\textup{alg}}^{i}}{1-q_{\textup{alg}}^{i}} \,
    \enorm{u_{\ell}^{k+1,\underline{i}} - u_{\ell}^{k,\underline{i}}} =
    \frac{q_{\textup{alg}}^{i}}{1-q_{\textup{alg}}^{i}} \,  \enorm{e_{\ell}^{k+1}}.
  \end{equation}
  With~\eqref{eq:aposteriori:alg}, $i_{\textup{min}} \le i \le \underline{i}[\ell, k+1$],
  and~\eqref{eq:choice:imin}, we derive
  \begin{align}
    &\edual{ u_{\ell}^{k+1, \star} -u_{\ell}^{k,\underline{i}}}{e^{k+1}_{{\ell}}}
    = \edual{ u_{\ell}^{k+1, \underline{i}} -u_{\ell}^{k,\underline{i}}}{e^{k+1}_{{\ell}}} + \edual{
      u_{\ell}^{k+1, \star}
    -u_{\ell}^{k+1,\underline{i}}}{e^{k+1}_{{\ell}}} \notag \\
    &  = \enorm{e^{k+1}_{{\ell}}}^2 + \edual{ u_{\ell}^{k+1, \star} -u_{\ell}^{k+1,\underline{i}}}{e^{k+1}_{{\ell}}}
    \ge \enorm{e^{k+1}_{{\ell}}}^2 - \enorm{ u_{\ell}^{k+1, \star} -u_{\ell}^{k+1,\underline{i}}}\,
    \enorm{e^{k+1}_{{\ell}}} \notag \\
    & \eqreff*{eq:aposteriori:alg}\ge \, \,
    \enorm{e^{k+1}_{{\ell}}} \,\Big[  \enorm{e^{k+1}_{{\ell}}} - \frac{q_{\textup{alg}}^{i}}{1-q_{\textup{alg}}^{i}} \,
    \enorm{e_{\ell}^{k+1}}\Big]
    = \Big( \frac{1-2 \, q_{\textup{alg}}^{i}}{1-q_{\textup{alg}}^{i}} \Big)\, \enorm{e^{k+1}_{{\ell}}}^2
    \eqreff{eq:choice:imin}\ge \frac{1}{2} \, \enorm{e^{k+1}_{{\ell}}}^2 \ge 0. \label{eq:f4:aux2}
  \end{align}
  With the local Lipschitz continuity~\eqref{eq:locally-lipschitz} and~\eqref{eq:critical}, it follows
  from~\eqref{eq:f4:aux:upper} that
  \begin{align*}
    \mathcal{E}(u^{k,\underline{i}}_{{\ell}}) - \mathcal{E}(u^{k+1,\underline{i}}_{{\ell}})
    \, \,&\eqreff*{eq:locally-lipschitz}\ge \, \,
    - \Big(\int_0^1 t L[5\tau] \, \textup{d}t \Big) \,  \enorm{e_{{\ell}}^{k+1}}^2 + \frac{1}{\delta} \,\edual{
    u_{\ell}^{k+1, \star} -u_{\ell}^{k,\underline{i}}}{e^{k+1}_{{\ell}}} \notag\\
    &\eqreff*{eq:f4:aux2}\ge \, \,
    \Big[ \frac{1}{2 \, \delta}  - \frac{L[5\tau]}{2} \Big]\,   \enorm{e^{k+1}_{{\ell}}}^2.
  \end{align*}
  Since $0 <\delta < 1/L[5\tau]$, the last  expression is positive.

  {\bf Step 5 (proof of upper bound in~(\ref{eq:f4})).}
  To derive the upper equivalence constant, we infer from Step~4 that
  \begin{align*}
    \edual{ u_{\ell}^{k+1, \star} -u_{\ell}^{k,\underline{i}}}{e^{k+1}_{{\ell}}}
    &\le \enorm{e^{k+1}_{{\ell}}}^2 + \enorm{ u_{\ell}^{k+1, \star} -u_{\ell}^{k+1,\underline{i}}}\,
    \enorm{e^{k+1}_{{\ell}}} \\
    &\eqreff*{eq:aposteriori:alg}\le
    \enorm{e^{k+1}_{{\ell}}} \,\Big[  \enorm{e^{k+1}_{{\ell}}} + \frac{q_{\textup{alg}}^{i}}{1-q_{\textup{alg}}^{i}} \,
    \enorm{e_{\ell}^{k+1}}\Big] = \Big( \frac{1}{1-q_{\textup{alg}}^{i}} \Big) \enorm{e^{k+1}_{{\ell}}}^2.
  \end{align*}
  Combined with Step~3, we obtain that
  \begin{align*}
    \mathcal{E}(u^{k,\underline{i}}_{{\ell}}) &- \mathcal{E}(u^{k+1,\underline{i}}_{{\ell}})
    \stackrel{\mathclap{\eqref{eq:f4:aux:upper}}}{=}  - \int_0^1 \! \dual{\mathcal{A}(u^{k,\underline{i}}_{{\ell}} + t
    \,e^{k+1}_{{\ell}}) \!- \! \mathcal{A} (u^{k,\underline{i}}_{{\ell}})}{e^{k+1}_{{\ell}}} \, \textup{d}t +
    \frac{1}{\delta}\,
    \edual{ u_{\ell}^{k+1, \star}\!-\! u_{\ell}^{k,\underline{i}}}{e^{k+1}_{{\ell}}}  \\
    &\eqreff*{eq:strongly-monotone}\le - \Big(\int_0^1 t \alpha \, \textup{d}t \Big) \,
    \enorm{e_{{\ell}}^{k+1}}^2 + \frac{1}{\delta} \,\edual{ u_{\ell}^{k+1, \star}
    -u_{\ell}^{k,\underline{i}}}{e^{k+1}_{{\ell}}}
    \\
    &\le \Big( \frac{1}{\delta \,(1-q_{\textup{alg}}^{i_{\textup{min}}})} - \frac{\alpha}{2}\Big) \,
    \enorm{e^{k+1}_{{\ell}}}^2.
  \end{align*}
  This concludes the proof.
\end{proof}
\begin{lemma}[energy contraction]\label{lemma:econtraction}
  Suppose the assumptions of Lemma~\ref{lemma:f4}. Recall $i_{\textup{min}} \in \mathbb{N}$
  from~\eqref{eq:imin}. Then, for $0 <
  \delta <  \min\{\frac{1}{L[5\tau]}, \frac{2\alpha}{L[2 \tau]^2}\}$, it holds that
  \begin{subequations}\label{eq:energycontr}
    \begin{equation}\label{eq:energycontr:ineq}
      0 \le \mathcal{E}(u^{k+1,\underline{i}}_{{\ell}})-\mathcal{E}(u^\star_{{\ell}}) \le
      q_{\mathcal{E}}[\delta,\tau]^2 \,
      [\mathcal{E}(u^{k,\underline{i}}_{{\ell}})-\mathcal{E}(u^\star_{{\ell}})]
    \end{equation}
    with the contraction constant
    \begin{equation}\label{eq:energycontr:const}
      0 \le q_{\mathcal{E}}[\delta,\tau]^2 \coloneqq  \, 1 - \Big( \frac{1}{\delta}  - L[5\tau] \Big)\,
      \frac{(1-q_{\textup{alg}}^{i_{\textup{min}}})^2 \, \delta^2  \alpha^2}{L[2\tau]} < 1.
    \end{equation}
  \end{subequations}
  We note that $q_{\mathcal{E}}[\delta, \tau] \to 1$ as $\delta \to 0$. In particular, it holds that
  \begin{equation}\label{eq:energycontr:triangle}
    (1- q_{\mathcal{E}}[\delta,\tau]^{2})\, \big[\mathcal{E}(u^{k,\underline{i}}_{\ell}) -
    \mathcal{E}(u_{\ell}^\star)\big] \le
    \mathcal{E}(u_{\ell}^{k,\underline{i}}) - \mathcal{E}(u_{\ell}^{k+1,\underline{i}}) \le
    \mathcal{E}(u^{k,\underline{i}}_{\ell}) - \mathcal{E}(u_{\ell}^\star).
  \end{equation}%
\end{lemma}
\begin{proof}
  First, we observe that
  \begin{align}\label{eq:helper:equiv}
    \begin{split}
      \alpha \, \enorm{u^\star_{\ell} - u^{k,\underline{i}}_{\ell}}^2 \, \,&\eqreff*{eq:strongly-monotone}\le \,
      \,\dual{\mathcal{A} u^\star_{\ell} - \mathcal{A} u^{k,\underline{i}}_{\ell}}{u^\star_{\ell} -
      u^{k,\underline{i}}_{\ell}}
      \stackrel{\eqref{eq:weakform:discrete}}{=}  \dual{F - \mathcal{A} u^{k,\underline{i}}_{\ell}}{u^\star_{\ell} -
      u^{k,\underline{i}}_{\ell}} \\
      & \stackrel{\mathclap{\eqref{eq:Zarantonello:iteration}}}{=} \,  \frac{1}{\delta}
      \,\edual{u^{k+1,\star}_{\ell} - u^{k,\underline{i}}_{\ell}}{u^\star_{\ell} - u^{k,\underline{i}}_{\ell}} \le
      \frac{1}{\delta} \, \enorm{u^{k+1,\star}_{\ell} - u^{k,\underline{i}}_{\ell}}\, \enorm{u^\star_{\ell} -
      u^{k,\underline{i}}_{\ell}}.
    \end{split}
  \end{align}
  The inverse triangle inequality and contraction~\eqref{eq:algebraContraction} of the algebraic solver prove that
  \begin{align}
    \begin{split}\label{eq:energyContraction:aux}
      \enorm{u_{\ell}^{k+1, \underline{i}} - u_{\ell}^{k, \underline{i}}}
      &\ge \enorm{u_{\ell}^{k+1, \star} - u_{\ell}^{k, \underline{i}}} -  \enorm{ u_{\ell}^{k+1,
      \star}-u_{\ell}^{k+1, \underline{i}}}\\
      &\eqreff*{eq:algebraContraction}{\ge}
      \, \, (1-q_{\textup{alg}}^{i_{\textup{min}}}) \, \enorm{u_{\ell}^{k+1, \star} - u_{\ell}^{k, \underline{i}}}
      \eqreff{eq:helper:equiv}\ge (1-q_{\textup{alg}}^{i_{\textup{min}}}) \, \delta \, \alpha\,
      \enorm{u_{\ell}^{\star} -
      u_{\ell}^{k, \underline{i}}}.
    \end{split}
  \end{align}
  Since $0 < \delta <  \min\{\frac{1}{L[5\tau]}, \frac{2\alpha}{L[2 \tau]^2}\}$, it follows that
  \begin{align*}
    0 \,&\stackrel{\mathclap{\eqref{eq:equivalence}}}{\le} \,\mathcal{E}(u^{k+1,\underline{i}}_{{\ell}}) -
    \mathcal{E}(u^{\star}_{{\ell}})
    =
    \mathcal{E}(u^{k,\underline{i}}_{{\ell}}) - \mathcal{E}(u^{\star}_{{\ell}}) -
    \big[\mathcal{E}(u^{k,\underline{i}}_{{\ell}}) -
    \mathcal{E}(u^{k+1,\underline{i}}_{{\ell}}) \big] \\
    &\stackrel{\mathclap{\eqref{eq:f4}}}{\le}
    \, \mathcal{E}(u^{k,\underline{i}}_{{\ell}}) - \mathcal{E}(u^{\star}_{{\ell}}) - \Big( \frac{1}{2\delta}
      - \frac{L[5\tau]}{2}
    \Big)\,  \enorm{u_{{\ell}}^{k+1,\underline{i}}-u_{{\ell}}^{k,\underline{i}}}^2 \\
    &  \stackrel{\mathclap{\eqref{eq:energyContraction:aux}}}{\le} \,
    \mathcal{E}(u^{k, \underline{i}}_{{\ell}}) - \mathcal{E}(u^{\star}_{{\ell}}) - \Big( \frac{1}{2\delta}
      - \frac{L[5\tau]}{2}
    \Big)\, (1-q_{\textup{alg}}^{i_{\textup{min}}})^2 \, \delta^2 \, \alpha^2 \,
    \enorm{u_{{\ell}}^{\star}-u_{{\ell}}^{k,\underline{i}}}^2 \\
    &   \stackrel{\mathclap{\eqref{eq:equivalence}}}{\le}
    \, \Big( 1 - \big[ 1 - \delta\,L[5\tau]\big]\, \frac{(1-q_{\textup{alg}}^{i_{\textup{min}}})^2 \,  \alpha^2\,
    \delta}{L[2\tau]}\Big)\,\big[\mathcal{E}(u^{k,\underline{i}}_{{\ell}}) -  \mathcal{E}(u^{\star}_{{\ell}})\big] \\
    &\eqqcolon q_{\mathcal{E}}[\delta, \tau]^2 \, \big[\mathcal{E}(u^{k,\underline{i}}_{{\ell}}) -
    \mathcal{E}(u^{\star}_{{\ell}})\big].
  \end{align*}
  We may rewrite $q_{\mathcal{E}}[\delta, \tau]^2 = 1 - C \delta + C\, L[5\tau] \,\delta^2 $ with $C =
  \frac{(1-q_{\textup{alg}}^{i_{\textup{min}}})^2 \,  \alpha^2}{L[2\tau]}$. Since $0 < \delta <
  \min\{\frac{1}{L[5\tau]},
  \frac{2\alpha}{L[2 \tau]^2}\} \le \frac{1}{L[5\tau]}$, we obtain that $0<q_{\mathcal{E}}[\delta,\tau] < 1$. This
  proves~\eqref{eq:energycontr}. The lower inequality in~\eqref{eq:energycontr:triangle} follows from the
  triangle inequality. The upper inequality in~\eqref{eq:energycontr:triangle} holds due to $0 \le
  \mathcal{E}(u_\ell^{k+1, \underline{i}})-\mathcal{E}(u_\ell^\star)$ and hence $\mathcal{E}(u_\ell^{k,
  \underline{i}})-\mathcal{E}(u_\ell^{k+1, \underline{i}}) =
  \mathcal{E}(u_\ell^{k, \underline{i}})-\mathcal{E}(u_\ell^\star) +
  \mathcal{E}(u_\ell^{\star})-\mathcal{E}(u_\ell^{k+1, \underline{i}}) \le
  \mathcal{E}(u_\ell^{k, \underline{i}})-\mathcal{E}(u_\ell^\star)$. This concludes the proof.
\end{proof}

\begin{proof}[\textbf{\textit{Proof of Theorem~\ref{theorem:uniformBoundedness}}}]
  The proof consists of four steps.

  \textbf{Step 1 (proof of~(\ref{crucial:mbounded})--(\ref{crucial:energyContraction}) for $\boldsymbol{k=0}$ and all
  $\boldsymbol{\ell \in \mathbb{N}_0}$).}
  Let $\ell\in \mathbb{N}_0$ with $\ell \le \underline{\ell}$ be arbitrary, but fixed. From the initial guess
  $u_0^{0,0}$ or
  Algorithm~\ref{algorithm:AILFEM}{\rm (I.c)} and $u_\ell^{0,\underline{i}} =u_\ell^{0,0}=
  u_{\ell-1}^{\underline{k},\underline{i}}$
  for any $\ell \in \mathbb{N}$, we have that $\enorm{u_\ell^{0,0}} \le 2M$ and \textsl{a~fortiori}
  $\enorm{u_\ell^{0,0}}\le\tau$. This proves~\eqref{crucial:mbounded} for $k=0$ and all $\ell \in \mathbb{N}_0$ with
  $\ell \le \underline{\ell}$ (even with the stronger bound $2M \le \tau$).

  In particular, we may apply Lemma~\ref{lemma:econtraction} to obtain that $\mathcal{E}(u_\ell^{1, \underline{i}}) -
  \mathcal{E}(u_\ell^\star) \le q_{\mathcal{E}}[\delta, \tau]^2 \, \big[ \mathcal{E}(u_\ell^{0,\underline{i}}) -
  \mathcal{E}(u_\ell^\star) \big]$,
  which proves~\eqref{crucial:energyContraction} for $k=0$ and $\ell \in \mathbb{N}_0$.

  \textbf{Step 2 (proof of~(\ref{crucial:mbounded})--(\ref{crucial:energyContraction}) for $\boldsymbol{k\ge0}$ and all
  $\boldsymbol{\ell \in \mathbb{N}_0}$).}
  Let $\ell\in \mathbb{N}_0$ with $\ell \le \underline{\ell}$. We argue by induction on $k$, where Step~1
  proves the base case
  $k=0$. Hence, we may assume that boundedness~\eqref{crucial:mbounded} holds for all $0\le k' \le k$.
  Lemma~\ref{lemma:econtraction} applied separately for all $0 \le k' \le k$ yields energy
  contraction~\eqref{eq:energycontr} for the indices $0 \le k' \le k$. Overall, we obtain that
  \begin{equation}\label{eq:induction:2}
    \mathcal{E}(u_\ell^{k+1, \underline{i}}) - \mathcal{E}(u_\ell^\star)
    \eqreff*{eq:energycontr}\le q_{\mathcal{E}}[\delta, \tau]^2 \, \big[ \mathcal{E}(u_\ell^{k,\underline{i}}) -
    \mathcal{E}(u_\ell^\star) \big]
    \eqreff{eq:energycontr}\le
    q_{\mathcal{E}}[\delta, \tau]^{2(k+1)} \, \big[ \mathcal{E}(u_\ell^{0,\underline{i}}) -
    \mathcal{E}(u_\ell^\star) \big],
  \end{equation}
  where we only used energy contraction~\eqref{eq:energycontr} for $0 \le k' \le k$, i.e., for indices that
  are covered by the induction hypothesis. From~\eqref{eq:induction:2}, $\enorm{u_\ell^\star}\le M$
  from~\eqref{eq:exact:bounded}, and $\enorm{u_\ell^{0, \underline{i}}} \le 2M$ and  $u_\ell^{0,\underline{i}} =
  u_\ell^{0,0}$ from Step~1, we obtain that
  \begin{align}
    \begin{split}\label{eq:step}
      \enorm{u_\ell^{k+1,\underline{i}}}
      &\le \enorm{u_\ell^\star} + \enorm{u_\ell^\star - u_\ell^{k+1,\underline{i}}} \\
      &\eqreff*{eq:equivalence}\le M+ \Big(\frac{2}{\alpha}\Big)^{1/2} \,
      \big[\mathcal{E}(u_\ell^{k+1,\underline{i}})-
      \mathcal{E}(u_\ell^\star)\big]^{1/2}\\
      &\eqreff*{eq:induction:2}\le M + q_{\mathcal{E}}^{k+1} \, \Big(\frac{2}{\alpha}\Big)^{1/2} \,
      \big[\mathcal{E}(u_\ell^{0,\underline{i}})- \mathcal{E}(u_\ell^\star)\big]^{1/2} \\
      & \eqreff*{eq:equivalence}\le
      M+
      q_{\mathcal{E}}^{k+1} \, \Big(\frac{L[3M]}{\alpha}\Big)^{1/2} \,\enorm{u_\ell^\star -
      u_\ell^{0,\underline{i}}} \le M+
      q_{\mathcal{E}}^{k+1} \, \Big(\frac{L[3M]}{\alpha}\Big)^{1/2} \,3M \le \tau.
    \end{split}
  \end{align}
  Thus, boundedness~\eqref{crucial:mbounded} is satisfied for $0 \le k' \le k+1$. Again, Lemma~8 yields
  energy contraction for $0 \le k' \le k+1$. This completes the induction argument and concludes
  that~\eqref{crucial:mbounded}--\eqref{crucial:energyContraction} hold for all $\ell \in \mathbb{N}_0$ and
  all $k \in \mathbb{N}_0$.

  \textbf{Step~3 (uniform boundedness).}
  Contraction of the algebraic solver~\eqref{eq:algebraContraction}, the straightforward estimate from the
  exact Zarantonello iteration~\eqref{eq:Zarantonello:iteration}, $\enorm{u^\star} \le M \le \tau$
  from~\eqref{eq:exact:bounded}, $\enorm{u_\ell^{k,0}} \le \tau$ from~\eqref{crucial:mbounded}, and the
  constraint $\delta < \min\{1/L[5\tau], 2 \alpha/L[2\tau]^2\}$ which ensures that $\delta L[2\tau] \le
  \delta L[5\tau] < 1$, yield that
  \begin{equation*}
    \enorm{u_\ell^{k,\star} - u_\ell^{k,0}}
    =  \enorm{\Phi_\ell(\delta; u_\ell^{k,0}) - u_\ell^{k,0}}
    \le \delta  \,\| F -\mathcal{A}(u_\ell^{k,0}) \|_{\mathcal{X}'}
    \eqreff*{eq:locally-lipschitz}\le \, \,
    \delta  \, L[2 \tau] \,\enorm{u^\star -u_\ell^{k,0}} < 2 \tau.
  \end{equation*}
  With $\enorm{u_\ell^{k, \star}} \le \enorm{u_\ell^{k, 0}} + \enorm{u_\ell^{k, \star}- u_\ell^{k, 0}} \le
  3 \tau$ owing to~\eqref{crucial:nestediteration}, it follows that
  \begin{align*}
    \enorm{u_\ell^{k,i}}
    \eqreff{eq:algebraContraction}\le \enorm{u_\ell^{k,\star}} + q_{\textup{alg}}^i \, \enorm{u_\ell^{k,\star} -
    u_\ell^{k,0}} \le 5\tau \quad \text{ for all } (\ell, k, i) \in \mathcal{Q}.
  \end{align*}

  \textbf{Step 4 (existence of $\boldsymbol{k_0}$).}
  Let $\ell\in \mathbb{N}_0$ with $\ell \le \underline{\ell}$. As in~\eqref{eq:step} from Step~2, we obtain
  \begin{equation*}
        \enorm{u_\ell^{k,\underline{i}}}
      \le M+
    q_{\mathcal{E}}^{k} \,\Big( \frac{L[3M]}{\alpha}\Big)^{1/2} \,3M.\label{eq:nested:aux}
  \end{equation*}
  Clearly, there exists a minimal integer $k_0=k_0[q_{\mathcal{E}}, \alpha, L[3M]]=k_0[\delta, \tau, \alpha, L[3M], M]
  \in \mathbb{N}$ such that, for all $k \ge k_0$, it holds that
  \[
    M+
    q_{\mathcal{E}}^{k} \,\Big( \frac{L[3M]}{\alpha}\Big)^{1/2} \,3M \le 2M.
  \]
  In particular, $k_0$ is independent of the mesh level $\ell$ and $\enorm{u_\ell^{k, \underline{i}}} \le 2M$ for all
  $k_0 \le k \le \underline{k}[\ell]$. This concludes the proof.
\end{proof}
\begin{remark}  \label{rem:f4andUniform}
  {\rm (i)} According to uniform boundedness~\eqref{eq:uniform:all}, all involved Lipschitz constants or
  stability constants are uniformly bounded by $L[10\tau]$ and $C_{\textup{stab}}[10\tau]$, respectively.

  {\rm (ii)} Under the assumption that $0 < \delta <  \min\{\frac{1}{L[5\tau]}, \frac{2\alpha}{L[2
  \tau]^2}\}$, energy contraction~\eqref{crucial:energyContraction} and the lower bound in the norm-energy
  equivalence~\eqref{eq:f4} are even equivalent, i.e.,
  \begin{equation*}
    \eqref{crucial:energyContraction} \quad \iff \quad     \eqref{eq:f4}.
  \end{equation*}
  To see this, recall that the proof of energy contraction~\eqref{crucial:energyContraction} in
  Lemma~\ref{lemma:econtraction} exploits~\eqref{eq:f4}. The converse implication is obtained as follows:
  First, energy contraction yields
  \begin{equation}
    \begin{aligned}
      \mathcal{E}(u^{k+1,\star}_\ell) - \mathcal{E}(u_\ell^\star)
      &\le
      q_{\mathcal{E}}^2 \big[   \mathcal{E}(u^{k,\underline{i}}_\ell) - \mathcal{E}(u_\ell^\star)\big]
      \\
      &=
      q_{\mathcal{E}}^2 \,\big\{ \big[
        \mathcal{E}(u^{k,\underline{i}}_\ell) - \mathcal{E}(u_\ell^{k+1, \star})\big] \!+\!
        \big[\mathcal{E}(u_\ell^{k+1, \star})-
      \mathcal{E}(u_\ell^\star)\big] \big\}
    \end{aligned}
  \end{equation}
  which gives rise to the \textsl{a~posteriori} estimate
  \begin{align}\label{eq:remark:energyDifference}
    0 \le   \mathcal{E}(u^{k+1,\star}_\ell) - \mathcal{E}(u_\ell^\star) \le
    \frac{q_{\mathcal{E}}^2}{1-q_{\mathcal{E}}^2} \, \big[
      \mathcal{E}(u_\ell^{k,
    \underline{i}}) - \mathcal{E}(u^{k+1,\star}_\ell)\big].
  \end{align}
  In particular, we note that the energy difference on the right-hand side is nonnegative. Exploiting uniform
  boundedness~\eqref{crucial:mbounded}, the last inequality yields that
  \begin{align*}
    \enorm{u_\ell^{k+1, \star} - u_\ell^{k, \underline{i}}}^2
    &\lesssim
    \enorm{u_\ell^{\star}- u_\ell^{k+1, \star}}^2 + \enorm{u_\ell^{\star} -u_\ell^{k, \underline{i}}}^2
    \\
    &\eqreff*{eq:Zarantonello:iteration}\le
    \big(1+(q_{\textup{Zar}}^{\star}[\delta, 2\tau])^2\big) \,  \enorm{u_\ell^{\star} - u_\ell^{k, \underline{i}}}^2\\
    &\eqreff*{eq:equivalence}\lesssim
    \big[\mathcal{E}(u_\ell^{k, \underline{i}}) - \mathcal{E}(u_\ell^{k+1,\star})\big] +
    \big[\mathcal{E}(u_\ell^{k+1,\star})
    -\mathcal{E}(u_\ell^{\star})\big]
    \\
    &\eqreff*{eq:remark:energyDifference}\le
    \frac{1}{1-q_{\mathcal{E}}^2} \,\big[\mathcal{E}(u_\ell^{k, \underline{i}}) - \mathcal{E}(u^{k+1,\star}_\ell)\big].
  \end{align*}
  This concludes the argument.
\end{remark}

\begin{remark}\label{remark:stoppingAndMark}
  {\rm (i)} The stopping criteria~\eqref{eq:i_stopping_criterion} and~\eqref{eq:k_stopping_criterion} read schematically
  \begin{equation*}
    [\mathtt{accuracy\,criterion}] \quad \mathtt{ AND } \quad [\mathtt{iteration\,criterion}].
  \end{equation*}

  {\rm (ii)} The accuracy criterion in~\eqref{eq:k_stopping_criterion} is heuristically motivated by the fact
  that the discretization error (estimated by $\eta_\ell(\cdot)$) shall dominate the linearization error
  \begin{equation}\label{eq:remark:linearization}
    \begin{aligned}
      \frac{\alpha}{2}\,\enorm{u_\ell^\star - u_\ell^{k+1, \underline{i}}}^2
      &\eqreff*{eq:equivalence}\le \,
      \mathcal{E}(u_\ell^{k+1,\underline{i}}) - \mathcal{E}(u_\ell^\star)
      \\
      &\eqreff*{crucial:energyContraction}\le \,
      \frac{q_{\mathcal{E}}^2}{1-q_{\mathcal{E}}^2}\,   \big[ \mathcal{E}(u_\ell^{k,\underline{i}}) -
      \mathcal{E}(u_\ell^{k+1,\underline{i}}) \big] \!
      \eqreff{eq:k_stopping_criterion}\lesssim  \! \lambda_{\textup{lin}}^2 \, \eta_\ell(u_\ell^{k+1, \underline{i}})^2.
    \end{aligned}
  \end{equation}
  This allows \textsl{a~posteriori} error control over the linearization error by means of computable energy
  differences.

  {\rm (iii)} The accuracy criterion~\eqref{eq:i_stopping_criterion} is satisfied given that the
  discretization and linearization error dominate the algebraic error in the sense of
  \begin{equation}\label{eq:algebraic:estimate}
    \begin{aligned}
      \enorm{u_\ell^{k,\star} - u_\ell^{k, i}}
      &\eqreff*{eq:algebraContraction}\le \,
      \frac{q_{\textup{alg}}}{1-q_{\textup{alg}}} \,\enorm{u_\ell^{k, i} - u_\ell^{k, i-1}}
      \\
      &\eqreff*{eq:i_stopping_criterion}\le \frac{q_{\textup{alg}}}{1-q_{\textup{alg}}}  \,
      \lambda_{\textup{alg}} \, \big[ \lambda_{\textup{lin}} \,
      \eta_\ell(u_\ell^{k, i}) + \enorm{u_\ell^{k,i} - u_\ell^{k,0}} \big].
    \end{aligned}
  \end{equation}
  Once the $i$-loop is stopped, the equivalence~\eqref{eq:f4} and nested iteration $u_\ell^{k,0} =
  u_\ell^{k-1, \underline{i}}$ yield $ \enorm{u_\ell^{k,\underline{i}} - u_\ell^{k,0}}^2 =
  \enorm{u_\ell^{k,\underline{i}} -
  u_\ell^{k-1,\underline{i}}}^2 \simeq \mathcal{E}(u_\ell^{k-1,\underline{i}}) - \mathcal{E}(u_\ell^{k,\underline{i}})$.
\end{remark}



\section{Full R-linear convergence}\label{sec:Rlinear}

We prove full R-linear convergence of Algorithm~\ref{algorithm:AILFEM} by adapting the analysis
of~\cite{hpsv2021,fps2023}. The new result extends~\cite[Theorem~13]{bbimp2022cost}, where an exact solve for
the Zarantonello iteration~\eqref{eq:Zarantonello:iteration} is supposed. The new proof is built on a
summability argument, but the stopping
criteria~\eqref{eq:i_stopping_criterion}--\eqref{eq:k_stopping_criterion} with iteration count criteria
require further analysis to prove full R-linear convergence even (and unlike~\cite{hpsv2021, fps2023}) for
arbitrary adaptivity parameters $0 < \theta \le 1$, $\lambda_{\textup{lin}}>0 $ and $\lambda_{\textup{alg}}> 0$.

\begin{theorem}[full R-linear convergence of Algorithm~\ref{algorithm:AILFEM}]\label{theorem:RLinearConvergence}
  Suppose the assumptions of Theorem~\ref{theorem:uniformBoundedness}. Suppose the axioms of
  adaptivity~\eqref{axiom:stability}--\eqref{axiom:reliability}. Let $\lambda_{\textup{lin}},
  \lambda_{\textup{alg}} > 0$, $0 < \theta
  \le 1$, $C_{\textup{mark}} \ge 1$, and $u_0^{0,0} \in \mathcal{X}_0$ with $\enorm{u_0^{0,0}} \le 2M$. Then,
  Algorithm~\ref{algorithm:AILFEM} guarantees full R-linear convergence of the quasi-error
  \begin{equation}\label{eq:quasi-error}
    \textup{H}_\ell^{k,i}
    \coloneqq
    \enorm{u_\ell^\star - u_\ell^{k,i}} + \enorm{u_\ell^{k,\star} - u_\ell^{k,i}}
    + \eta_\ell(u_\ell^{k,i}),
  \end{equation}
  i.e., there exist constants $0 < q_{\textup{lin}} < 1$ and $C_{\textup{lin}} > 0$ such that
  \begin{equation}\label{eq:RLinearConvergence}
    \!\! \!\textup{H}_\ell^{k,i}
    \! \le \!
    C_{\textup{lin}} q_{\textup{lin}}^{|\ell,k,i| - |\ell'\!,k'\!,i'|} \, \textup{H}_{\ell'}^{k',i'}
    \text{ for all }  \!(\ell'\!,k'\!,i'),(\ell,k,i) \in \mathcal{Q}
    \text{ with } |\ell'\!,k'\!,i'| \! < \! |\ell,k,i|.
  \end{equation}
  The constant $q_{\textup{lin}}$ depends only on $\theta$, $q_{\textup{red}}$ from~\eqref{axiom:reduction},
  $q_{\textup{Zar}}^{\star}[\delta,
  2\tau]$ from Proposition~\ref{prop:zarantonello}, $q_{\mathcal{E}}$ from Theorem~\ref{theorem:uniformBoundedness}, and
  $q_{\textup{alg}}$ from~\eqref{eq:algebraContraction}. The constant $C_{\textup{lin}}$ depends only on $M$, $\alpha$,
  $C_{\textup{Céa}}[2M]$, $q_{\textup{Zar}}^{\star}[\delta; 2\tau]$, $\lambda_{\textup{lin}}$, $q_{\textup{alg}}$,
  $\lambda_{\textup{alg}}$, $C_{\textup{rel}}$, $C_{\textup{stab}}[10\tau]$, and $i_{\textup{min}}$.
\end{theorem}
\begin{proof}[\bf \textit{Proof of Theorem~\ref{theorem:RLinearConvergence}}]
  The proof is split into seven steps.

  {\bf Step 1 (equivalences of quasi-error quantities).}
  Throughout the proof, we approach $\textup{H}_\ell^{k, i}$ from~\eqref{eq:quasi-error}
  after introducing auxiliary quantities such as
  \begin{equation}\label{eq1:proof}
    \textup{H}_\ell^k
    \coloneqq
    [\mathcal{E}(u_\ell^{k,\underline{i}})- \mathcal{E}(u_\ell^\star)]^{1/2}
    + \eta_\ell(u_\ell^{k,\underline{i}})
    \quad \text{for all } (\ell,k,\underline{i}) \in \mathcal{Q}
  \end{equation}
  and
  \begin{equation}\label{eq:quasi-error:discrete}
    \textup{H}_\ell \coloneqq [ \mathcal{E}(u_\ell^{\underline{k}, \underline{i}}) -
    \mathcal{E}(u_\ell^\star)]^{1/2} +\gamma \,
    \eta_\ell(u_\ell^{\underline{k}, \underline{i}}) \eqreff{eq1:proof}\simeq  \textup{H}_\ell^{\underline{k}} \quad
    \text{ for all } (\ell,
    \underline{k}, \underline{i}) \in \mathcal{Q},
  \end{equation}
  where $0 < \gamma < 1$ is a free parameter to be fixed later in~\eqref{eq:step3:2} below.
  In the following, we show that $\textup{H}_\ell^{\underline{k}, \underline{i}} \simeq \textup{H}_\ell^{\underline{k}}
  \eqreff{eq:quasi-error:discrete}\simeq \textup{H}_\ell$.
  First, note that the equivalence of energy and norm from~\eqref{eq:equivalence} (with $L[2\tau]$ from
  boundedness~\eqref{crucial:mbounded} and~\eqref{eq:exact:bounded}) yields that
  \begin{equation}\label{eq2:proof}
    \textup{H}_\ell^{k}
    \le
    \textup{H}_\ell^{k} + \enorm{u_\ell^{{k},\star} - u_\ell^{{k},\underline{i}}}
    \eqreff{eq:equivalence}\simeq
    \textup{H}_\ell^{{k},\underline{i}}
    \quad \text{ for all } (\ell,{k},\underline{i}) \in \mathcal{Q}.
  \end{equation}
  The \textsl{a~posteriori} estimate~\eqref{eq:algebraic:estimate} for the algebraic solver from
  Remark~\ref{remark:stoppingAndMark}{\rm(iii)}, norm-energy equivalence~\eqref{eq:f4}, and the stopping
  criterion~\eqref{eq:k_stopping_criterion} show that
  \begin{align*}
    \enorm{u_\ell^{\underline{k},\star} - u_\ell^{\underline{k},\underline{i}}}
    &\eqreff{eq:algebraic:estimate}\le
    \frac{q_{\textup{alg}}}{1-q_{\textup{alg}}}  \, \lambda_{\textup{alg}} \, \big[ \lambda_{\textup{lin}} \,
      \eta_\ell(u_\ell^{\underline{k}, \underline{i}}) +
    \enorm{u_\ell^{\underline{k},\underline{i}} - u_\ell^{\underline{k},0}} \big]
    \\
    &\eqreff{eq:f4}\lesssim
    \eta_\ell(u_\ell^{\underline{k}, \underline{i}}) + \big[\mathcal{E}(u_\ell^{k,0}) -
      \mathcal{E}(u_\ell^{\underline{k},
    \underline{i}})\big]^{1/2}
    \eqreff{eq:k_stopping_criterion}\lesssim
    \eta_\ell(u_\ell^{\underline{k},\underline{i}})
    \le
    \textup{H}_{\ell}^{\underline{k}}.
  \end{align*}
  With~\eqref{eq2:proof}, we conclude that $\textup{H}_\ell \simeq \textup{H}_\ell^{\underline{k}} \simeq
  \textup{H}_\ell^{\underline{k},\underline{i}}$.

  {\bf Step 2 (estimator reduction).}
  The axioms~\eqref{axiom:stability}--\eqref{axiom:reduction} and Dörfler marking~\eqref{eq:doerfler} prove
  the estimator reduction estimate (cf., e.g.,~\cite[Equation~(52)]{ghps2021})
  \begin{equation}\label{eq:estimator-reduction}
    \eta_{\ell+1}(u_{\ell+1}^{\underline{k}, \underline{i}})
    \le
    q_\theta \, \eta_\ell(u_\ell^{\underline{k}, \underline{i}}) + C_{\textup{stab}}[4M] \,
    \enorm{u_{\ell+1}^{\underline{k}, \underline{i}} -
    u_\ell^{\underline{k}, \underline{i}}}
    \quad \text{ for all } \ell \in \mathbb{N}_0,
  \end{equation}
  where $4M$ stems from nested iteration~\eqref{crucial:nestediteration} from
  Theorem~\ref{theorem:uniformBoundedness}. Moreover, the triangle inequality, the
  equivalence~\eqref{eq:equivalence}, and energy contraction~\eqref{crucial:energyContraction} give that
  \begin{align*}
    \enorm{u_{\ell+1}^{\underline{k}, \underline{i}} - u_\ell^{\underline{k}, \underline{i}}}
    & \le
    \enorm{u_{\ell+1}^{\star}-u_{\ell+1}^{\underline{k}, \underline{i}}}
    +
    \enorm{u_{\ell+1}^{\star}-u_{\ell}^{\underline{k}, \underline{i}}}   \\
    & \eqreff*{eq:equivalence}\le \Big( \frac{2}{\alpha} \Big)^{1/2}   \,
    \big[ \mathcal{E}(u_{\ell+1}^{\underline{k}, \underline{i}})-  \mathcal{E}(u_{\ell+1}^{\star})\big]^{1/2}  +
    \Big( \frac{2}{\alpha}
    \Big)^{1/2} \,  \big[\mathcal{E}(u_{\ell}^{\underline{k}, \underline{i}})-
    \mathcal{E}(u_{\ell+1}^{\star})\big]^{1/2}
    \\
    \, \,&\eqreff*{crucial:energyContraction}\le  \, \,
    (1+q_{\mathcal{E}}^{\underline{k}[\ell+1]}) \,\Big( \frac{2}{\alpha} \Big)^{1/2} \,
    \big[\mathcal{E}(u_{\ell}^{\underline{k}, \underline{i}})-
    \mathcal{E}(u_{\ell+1}^{\star})\big]^{1/2}.
  \end{align*}
  Combined with the estimator reduction estimate~\eqref{eq:estimator-reduction} and with $1+q_{\mathcal{E}} < 2$, we
  obtain with $C_1 \coloneqq 2\,(2/\alpha)^{1/2}\, C_{\textup{stab}}[4M]$ that
  \begin{equation}\label{eq:estimatorReductionFinal}
    \eta_{\ell+1}(u_{\ell+1}^{\underline{k}, \underline{i}})
    \le
    q_\theta \, \eta_\ell(u_\ell^{\underline{k}, \underline{i}}) +C_1 \,
    \big[\mathcal{E}(u_{\ell}^{\underline{k}, \underline{i}})-
    \mathcal{E}(u_{\ell+1}^{\star})\big]^{1/2}
    \quad \text{ for all } 0 \le \ell < \underline{\ell}.
  \end{equation}

  \textbf{Step~3 (tail summability with respect to $\boldsymbol{\ell}$).}
  Since $1 \le \underline{k}[\ell+1]$, nested iteration $u_{\ell+1}^{0, \underline{i}}=u_\ell^{\underline{k},
  \underline{i}}$
  proves that
  \begin{align}\label{eq:step3:1}
    \mkern-10mu  \textup{H}_{\ell+1} &\eqreff*{eq:quasi-error:discrete}=\,
    \big[\mathcal{E}(u_{\ell+1}^{\underline{k},
      \underline{i}})-
    \mathcal{E}(u_{\ell+1}^{\star})\big]^{1/2} + \gamma \,
    \eta_{\ell+1}(u_{\ell+1}^{\underline{k},\underline{i}})
    \\
    &\eqreff*{crucial:energyContraction}\le
    q_{\mathcal{E}} \, \big[\mathcal{E}(u_{\ell}^{\underline{k}, \underline{i}})-
    \mathcal{E}(u_{\ell+1}^{\star})\big]^{1/2}
    +  \, \gamma \, \eta_{\ell+1}(u_{\ell+1}^{\underline{k},\underline{i}}) \notag
    \\
    &\eqreff*{eq:estimatorReductionFinal}\le
    \bigl( q_{\mathcal{E}} + C_1 \, \gamma \bigr) \,
    \big[\mathcal{E}(u_{\ell}^{\underline{k}, \underline{i}})-  \mathcal{E}(u_{\ell+1}^{\star})\big]^{1/2}
    + q_\theta \,  \gamma \, \eta_\ell(u_\ell^{\underline{k},\underline{i}}) \notag
    \\
    &\le
    \max \bigl\{ q_{\mathcal{E}}\! + \!C_1 \,  \gamma,\, q_\theta \bigr\} \,
    \bigl( \big[\mathcal{E}(u_{\ell}^{\underline{k},
      \underline{i}})\!-\!\mathcal{E}(u_{\ell+1}^{\star})\big]^{1/2}\! +\! \gamma \,
    \eta_\ell(u_\ell^{\underline{k},\underline{i}}) \bigr)
    \text{ for all } (\ell+1,\underline{k},\underline{i}) \in \mathcal{Q}.
  \end{align}
  With $0 < q_\theta < 1$, we choose $0 < \gamma < (1-q_{\mathcal{E}})/C_1<1$ to guarantee that
  \begin{equation}\label{eq:step3:2}
    0 < \widetilde{q}
    \coloneqq
    \max \bigl\{ q_{\mathcal{E}}+ C_1 \,  \gamma,\, q_\theta  \bigr\} < 1.
  \end{equation}
  With the triangle inequality,~\eqref{eq:step3:1} leads us to
  \begin{align}
    \begin{split}\label{eq:requirement1}
      a_{\ell+1}
      &\, \coloneqq \,
      \big[\mathcal{E}(u_{\ell+1}^{\underline{k}, \underline{i}})-\mathcal{E}(u_{\ell+1}^{\star})\big]^{1/2}
      + \gamma \,
      \eta_{\ell+1}(u_{\ell+1}^{\underline{k},\underline{i}})
      \\&
      \, \eqreff*{eq:step3:1}\le \,
      \widetilde{q} \,
      \bigl( \big[\mathcal{E}(u_{\ell}^{\underline{k},
        \underline{i}})-\mathcal{E}(u_{\ell}^{\star})\big]^{1/2} + \gamma \,
      \eta_\ell(u_\ell^{\underline{k},\underline{i}}) \bigr)
      + \widetilde{q} \, \big[\mathcal{E}(u_{\ell}^{\star})-\mathcal{E}(u_{\ell+1}^{\star})\big]^{1/2}
      \\&
      \, \eqqcolon \,
      \widetilde{q} \, a_\ell + b_\ell
      \quad \text{ for all } (\ell,\underline{k},\underline{i}) \in \mathcal{Q}.
    \end{split}
  \end{align}
  By exploiting the equivalence~\eqref{eq:equivalence} and stability~\eqref{axiom:stability} (since all
    $u_\ell^{\underline{k}, \underline{i}}$ are uniformly bounded by nested
  iteration~\eqref{crucial:nestediteration}), the
  C\'ea lemma~\eqref{eq:cea}, and reliability~\eqref{axiom:reliability} prove that
  \begin{align}\label{eq:cea+reliability}
    \begin{split}
      &\mkern-12mu \big[\mathcal{E}(u_{\ell'}^{\star})-\mathcal{E}(u_{\ell''}^{\star})\big]^{1/2}
      \eqreff{eq:equivalence}\simeq \enorm{u_{\ell''}^\star - u_{\ell'}^\star}
      \eqreff*{eq:cea}\lesssim
      \enorm{u^\star - u_{\ell}^\star}
      \,  \eqreff*{axiom:reliability}\lesssim\,
      \eta_\ell(u_\ell^\star)
      \,  \eqreff*{axiom:stability}\lesssim
      \,     \enorm{u_\ell^\star - u_\ell^{\underline{k}, \underline{i}}} + \eta_\ell(u_\ell^{\underline{k},
      \underline{i}})  \\
      &
      \, \, \eqreff{eq:equivalence}{\simeq}
      \big[ \mathcal{E}(u_\ell^{\underline{k}, \underline{i}})- \mathcal{E}(u_\ell^\star) \big]^{1/2} +
      \eta_\ell(u_\ell^{\underline{k}, \underline{i}})
      \simeq a_\ell
      \text{ for all } \ell \le \ell' \le \ell'' \le \underline{\ell}
      \text{ with} (\ell,\underline{k}, \underline{i}) \in \mathcal{Q}.
    \end{split}
  \end{align}
  Hence, we infer that $b_{\ell+N} \lesssim a_\ell$ for all $0 \le \ell \le \ell+N \le \underline{\ell}$ with
  $(\ell,\underline{k},\underline{i}) \in
  \mathcal{Q}$, where the hidden stability constant $C_{\textup{stab}}[3M]$ depends on $3M$ due
  to~\eqref{eq:exact:bounded} and
  nested iteration~\eqref{crucial:nestediteration}.

  The energy $\mathcal{E}$ from~\eqref{eq:potential} (and its Pythagorean identity that leads to a telescoping sum)
  as well as the minimization property~\eqref{eq:emin} for $\mathcal{X}_{H}= \mathcal{X}$ allow for the estimate
  \begin{align}\label{eq:single:orthogonality}
    \begin{split}
      \sum_{\ell' = \ell}^{{\ell+N-1}}& b_{\ell'}^2
      \lesssim
      \sum_{\ell' = \ell}^{{\underline{\ell}-1}}\big[ \mathcal{E}(u_{\ell'}^\star)-\mathcal{E}(u_{\ell'+1}^\star) \big]
      \le
      \mathcal{E}(u_\ell^\star)-\mathcal{E}(u_{\underline{\ell}}^\star)
      \eqreff{eq:emin}\le
      \mathcal{E}(u_{\ell}^\star)-\mathcal{E}(u^\star)  \\
      &
      \eqreff*{eq:equivalence}\le
      \frac{L[2M]}{2} \, \enorm{u^\star - u_{\ell}^\star}^2
      \eqreff{axiom:reliability}\le C_{\textup{rel}}^2 \, \frac{L[2M]}{2} \, \eta_{\ell}(u_{\ell}^\star)^2
      \eqreff*{eq:cea+reliability}\lesssim  a_\ell^2
      {\text{ for all } 0 \le \ell < \ell+N \le \underline{\ell},}
    \end{split}
  \end{align}
  where the hidden stability constant $C_{\textup{stab}}$ depends on $3M$ due to~\eqref{eq:exact:bounded} and nested
  iteration~\eqref{crucial:nestediteration}.

  With~\eqref{eq:requirement1}--\eqref{eq:single:orthogonality}, the assumptions for the tail summability
  criterion from~\cite[Lemma~6]{fps2023} are met. We thus conclude tail summability of $\textup{H}_{\ell+1} \simeq
  \textup{H}_\ell^{\underline{k}} \simeq a_\ell$, i.e.,
  \begin{align}\label{eq:single:summability-ell}
    \boxed{\sum_{\ell' = \ell+1}^{\underline{\ell} - 1} \textup{H}_{\ell'}^{\underline{k}}
      \lesssim \textup{H}_\ell^{\underline{k}}
    \quad \text{ for all } 0 \le \ell < \underline{\ell}.}
  \end{align}

  \textbf{Step~4 (quasi-contraction in $\boldsymbol{k}$).}
  We distinguish three cases.

  \noindent{\bf Case 4.1: Evaluation of~(\ref{eq:k_stopping_criterion}) yields $\boldsymbol{\mathtt{TRUE}\, \land
  \,\mathtt{FALSE}}$.}
  This gives rise to
  \begin{align*}
    2M \eqreff{eq:k_stopping_criterion}< \enorm{u_\ell^{k, \underline{i}}} \le \enorm{u_\ell^\star}+
    \enorm{u_\ell^\star -u_\ell^{k, \underline{i}}} \eqreff{eq:exact:bounded}\le M + \enorm{u_\ell^\star
    -u_\ell^{k, \underline{i}}}
  \end{align*}
  and hence, we conclude that $ M < \enorm{u_\ell^\star -u_\ell^{k, \underline{i}}}$. Thus,
  \begin{align}
    \begin{split}\label{eq:k:counter:aux1}
      1= \frac{M}{M} < \frac{\enorm{u_\ell^\star -u_\ell^{k, \underline{i}}}}{M}
      \,&\eqreff*{eq:equivalence}\le \, \frac{1}{M} \,\Big(\frac{2}{\alpha} \Big)^{1/2} \,
      \big[\mathcal{E}(u_\ell^{k,\underline{i}}) - \mathcal{E}(u_\ell^\star)\big]^{1/2}
      \\&\eqreff*{crucial:energyContraction}\le \,
      \frac{q_{\mathcal{E}}}{M} \,\,\Big(\frac{2}{\alpha} \Big)^{1/2} \, \big[\mathcal{E}(u_\ell^{k-1,\underline{i}}) -
      \mathcal{E}(u_\ell^\star)\big]^{1/2}.
    \end{split}
  \end{align}
  We recall from~\eqref{eq:exact:bounded} that $\enorm{u_\ell^\star} \le M$ and $\enorm{u_\ell^\star -
  u_0^\star} \le 2M$ independently of $\ell$. Moreover, there holds quasi-monotonicity of the estimators in
  the sense that
  \begin{equation}\label{eq:monotonicityExact}
    \eta_\ell(u_\ell^\star) \le C_{\textup{mon}} \, \eta_0(u_0^\star) \quad \text{ with } C_{\textup{mon}} =
    \big[ 2 + 8 \,
    C_{\textup{stab}}[2M]^2 (1 + C_{\textup{Céa}}[2M]^2) \, C_{\textup{rel}}^2 \big]^{1/2};
  \end{equation}
  cf.~\cite[Lemma~3.6]{axioms} or~\cite[Equation~(42)]{bbimp2022cost} for the locally Lipschitz continuous
  setting. In particular, estimate~\eqref{eq:monotonicityExact} holds also for the discrete limit space
  $\mathcal{X}_{\underline{\ell}} \coloneqq \mathrm{closure}\big(\bigcup_{\ell=0}^{\underline{\ell}}
  \mathcal{X}_\ell \big)$. Additionally, we note
  that the estimate~\eqref{eq:monotonicityExact} admits
  \begin{equation}\label{eq:k:counter:aux2}
    \eta_\ell(u_\ell^\star)\, \eqreff*{eq:monotonicityExact}\le\, C_{\textup{mon}} \, \eta_0(u_0^\star)
    \,\eqreff*{axiom:stability}\le\, C_{\textup{mon}} \, \eta_0(0) + C_{\textup{mon}} \,
    C_{\textup{stab}}[M]\,\enorm{u_0^\star}
    \eqreff{eq:k:counter:aux1}\lesssim \big[\mathcal{E}(u_\ell^{k-1,\underline{i}}) -
    \mathcal{E}(u_\ell^{\star})\big]^{1/2}.
  \end{equation}
  The estimate~\eqref{eq:k:counter:aux2}, stability~\eqref{axiom:stability} with stability constant
  $C_{\textup{stab}}[2\tau]$ due to~\eqref{crucial:mbounded} and~\eqref{eq:exact:bounded}, and energy
  contraction~\eqref{crucial:energyContraction} yield that
  \begin{align}
    \begin{split}\label{eq:k:counter:aux3}
      \eta_\ell(u_\ell^{k, \underline{i}})
      \, \,&\eqreff*{axiom:stability}\le  \, \,
      \eta_\ell(u_\ell^{\star}) + C_{\textup{stab}}[2\tau] \, \enorm{u_\ell^\star - u_\ell^{k,\underline{i}}}
      \,\,   \eqreff*{eq:equivalence}\lesssim \, \,
      \eta_\ell(u_\ell^{\star}) + \big[\mathcal{E}(u_\ell^{k,\underline{i}}) -
      \mathcal{E}(u_\ell^{\star})\big]^{1/2} \\
      & \eqreff*{eq:k:counter:aux2}\lesssim \, \, \big[\mathcal{E}(u_\ell^{k-1,\underline{i}}) -
      \mathcal{E}(u_\ell^{\star})\big]^{1/2} +  \big[\mathcal{E}(u_\ell^{k,\underline{i}}) -
      \mathcal{E}(u_\ell^{\star})\big]^{1/2}
      \eqreff{crucial:energyContraction}\lesssim   \big[\mathcal{E}(u_\ell^{k-1,\underline{i}}) -
      \mathcal{E}(u_\ell^{\star})\big]^{1/2}.
    \end{split}
  \end{align}
  For $0 \le k' < k < \underline{k}[\ell]$, the definition~\eqref{eq1:proof}, energy
  contraction~\eqref{crucial:energyContraction}, and~\eqref{eq:k:counter:aux3} prove
  \begin{align}
    \begin{split}\label{eq:k:iterationCount}
      \textup{H}_\ell^k
      &\eqreff{crucial:energyContraction}\lesssim
      q_{\mathcal{E}} \, \big[\mathcal{E}(u_\ell^{k-1,\underline{i}}) -  \mathcal{E}(u_\ell^{\star})\big]^{1/2} +
      \eta_\ell(u_\ell^{k,\underline{i}})
      \eqreff{eq:k:counter:aux3}\lesssim
      \, \big[\mathcal{E}(u_\ell^{k-1,\underline{i}}) -  \mathcal{E}(u_\ell^{\star})\big]^{1/2} \\
      &\eqreff{crucial:energyContraction}\lesssim
      q_{\mathcal{E}}^{(k-1) - k'} \, \big[ \mathcal{E}(u_\ell^{k', \underline{i}}) -
      \mathcal{E}(u_\ell^\star)\big]^{1/2}
      \eqreff{eq1:proof}\lesssim
      q_{\mathcal{E}}^{k-k'} \, \textup{H}_\ell^{k'}.
    \end{split}
  \end{align}
  This concludes Case~4.1. \hfill $\diamond$

  \noindent {\bf Case 4.2: Evaluation of~(\ref{eq:k_stopping_criterion}) yields $\boldsymbol{\mathtt{FALSE}\, \land
  \,\mathtt{FALSE}}$ or $\boldsymbol{\mathtt{FALSE}\, \land \,\mathtt{TRUE}}$.} For $0 \le k' < k <
  \underline{k}[\ell]$, the
  definition~\eqref{eq1:proof}, the failure of the accuracy condition in the stopping criterion for the
  inexact Zarantonello linearization~\eqref{eq:k_stopping_criterion}, energy minimization~\eqref{eq:emin},
  and energy contraction~\eqref{crucial:energyContraction} prove that
  \begin{align}
    \begin{split}\label{eq:k:accuracy}
      \textup{H}_\ell^k
      \, \,  &  \eqreff*{eq:k_stopping_criterion}<  \, \,
      \big[\mathcal{E}(u_\ell^{k,\underline{i}}) -  \mathcal{E}(u_\ell^{\star})\big]^{1/2} +
      \lambda_{\textup{lin}}^{-1} \,
      \big[\mathcal{E}(u_\ell^{k-1,\underline{i}}) -  \mathcal{E}(u_\ell^{k,\underline{i}})\big]^{1/2} \\
      &
      \stackrel{\mathclap{\eqref{eq:emin},~\eqref{crucial:energyContraction}}}\lesssim  \, \, \,  \, \,
      \big[\mathcal{E}(u_\ell^{k-1,\underline{i}}) -  \mathcal{E}(u_\ell^{\star})\big]^{1/2}
      \eqreff{crucial:energyContraction}\lesssim
      q_{\mathcal{E}}^{(k-1)-k'} \, \big[\mathcal{E}(u_\ell^{k',\underline{i}}) -
      \mathcal{E}(u_\ell^{\star})\big]^{1/2}
      \eqreff{eq1:proof}\lesssim
      q_{\mathcal{E}}^{k-k'} \, \textup{H}_\ell^{k'}.
    \end{split}
  \end{align}
  This concludes Case~4.2. \hfill $\diamond$

  \noindent {\bf Case 4.3: Evaluation of~(\ref{eq:k_stopping_criterion}) yields $\boldsymbol{\mathtt{TRUE}\, \land
  \,\mathtt{TRUE}}$.} The equivalence~\eqref{eq:f4}, boundedness~\eqref{crucial:mbounded}, and energy
  minimization~\eqref{eq:emin} prove that
  \begin{align}\label{eq1:step7}
    \begin{split}
      \textup{H}_\ell^{\underline{k}}
      \,\, &\eqreff*{axiom:stability}\lesssim \, \,
      \big[ \mathcal{E}(u_\ell^{\underline{k},\underline{i}}) - \mathcal{E}(u_\ell^{\star}) \big]^{1/2}
      + \enorm{u_\ell^{\underline{k},\underline{i}} - u_\ell^{\underline{k}-1,\underline{i}}}
      + \eta_\ell(u_\ell^{\underline{k}-1,\underline{i}})
      \\& \eqreff*{eq:f4}\lesssim \, \,
      \textup{H}_\ell^{\underline{k}-1} + \big[  \mathcal{E}(u_\ell^{\underline{k}-1,\underline{i}}) -
      \mathcal{E}(u_\ell^{\underline{k},\underline{i}})\big]^{1/2}
      \eqreff{eq:emin}\le 2 \, \textup{H}_\ell^{\underline{k}-1}
      \quad \text{for all } (\ell,\underline{k},\underline{i}) \in \mathcal{Q}.
    \end{split}
  \end{align}
  Since $k=\underline{k}[\ell]-1$ is covered by Case~4.1 or Case~4.2, estimate~\eqref{eq1:step7} leads to
  \begin{equation}\label{eq:k:met}
    \textup{H}_\ell^{\underline{k}}
    \eqreff{eq1:step7}\lesssim \frac{q_{\mathcal{E}}}{q_{\mathcal{E}}}\, \textup{H}_\ell^{\underline{k}-1} \lesssim
    q_{\mathcal{E}} \, \textup{H}_\ell^{\underline{k}-1}
    \,\, \,\, \,\stackrel{\mathclap{\eqref{eq:k:iterationCount},~\eqref{eq:k:accuracy}}}\lesssim \, \,\, \,
    q_{\mathcal{E}}^{\underline{k}[\ell]-k'} \textup{H}_\ell^{k', \underline{i}}.
  \end{equation}
  This concludes Case~4.3. \hfill $\diamond$

  Overall, the
  estimates~\eqref{eq:k:iterationCount}--\eqref{eq:k:accuracy} and~\eqref{eq:k:met} result in
  \begin{equation}\label{eq2:step7}
    \boxed{
      \textup{H}_\ell^k
      \lesssim q_{\mathcal{E}}^{\,k-k'} \textup{H}_\ell^{k'}
      \quad \text{for all } (\ell,k,j) \in \mathcal{Q} \text{ with } 0 \le k' \le k \le \underline{k}[\ell],
    }
  \end{equation}
  where the hidden constant depends only on $M$, $C_{\textup{stab}}[2\tau]$, $\alpha$, $L[2M]$,
  $C_{\textup{Céa}}[2M]$, $C_{\textup{rel}}$,
  $\lambda_{\textup{lin}}$, and $q_{\mathcal{E}}$.
  Furthermore, we recall from~\eqref{eq:cea+reliability} that
  $\big[ \mathcal{E}(u_{\ell-1}^\star)-\mathcal{E}(u_{\ell}^\star)\big]^{1/2}
  \lesssim \textup{H}_{\ell-1}^{\underline{k}}$.
  Together with nested iteration $u_{\ell-1}^{\underline{k},\underline{i}} =
  u_\ell^{0,\underline{i}}=u_\ell^{0,\star}$, this yields that
  \begin{equation*}
    \textup{H}_\ell^0 = \big[ \mathcal{E}(u_{\ell-1}^{\underline{k},\underline{i}}) -
    \mathcal{E}(u_\ell^\star)\big]^{1/2} +
    \eta_\ell(u_{\ell-1}^{\underline{k},\underline{i}})
    \lesssim \big[ \mathcal{E}(u_{\ell-1}^{\star}) - \mathcal{E}(u_\ell^\star)\big]^{1/2} +
    \textup{H}_{\ell-1}^{\underline{k}} \le
    \textup{H}_{\ell-1}^{\underline{k}}
  \end{equation*}
  and thus
  \begin{align}\label{eq3:step7}
    \boxed{
      \textup{H}_\ell^0
      \lesssim
      \textup{H}_{\ell-1}^{\underline{k}}
      \quad \text{ for all } (\ell,0,0) \in \mathcal{Q} \text{ with } \ell \ge 1.
    }
  \end{align}

  \textbf{Step~5 (tail summability with respect to $\boldsymbol{\ell}$ and $\boldsymbol{k}$).}
  The estimates~\eqref{eq2:step7}--\eqref{eq3:step7} from Step~4 as well as~\eqref{eq:single:summability-ell}
  from Step~3 and the geometric series prove that
  \begin{align}\label{eq:step8}
    \begin{split}
      &\sum_{\substack{(\ell',k',\underline{i}) \in \mathcal{Q} \\ |\ell',k',\underline{i}| >
      |\ell,k,\underline{i}|}} \textup{H}_{\ell'}^{k'}
      =
      \sum_{k' = k+1}^{\underline{k}[\ell]} \textup{H}_\ell^{k'}
      + \sum_{\ell' = \ell+1}^{\underline{\ell}} \sum_{k'=0}^{\underline{k}[\ell']} \textup{H}_{\ell'}^{k'}
      \eqreff{eq2:step7}\lesssim
      \textup{H}_\ell^k
      + \sum_{\ell' = \ell+1}^{\underline{\ell}} \textup{H}_{\ell'}^0
      \\& \quad
      \eqreff{eq3:step7}\lesssim
      \textup{H}_\ell^k
      + \sum_{\ell' = \ell}^{\underline{\ell}-1} \textup{H}_{\ell'}^{\underline{k}}
      \eqreff{eq:single:summability-ell}
      \lesssim
      \textup{H}_\ell^k
      + \textup{H}_\ell^{\underline{k}}
      \eqreff{eq2:step7}\lesssim
      \textup{H}_\ell^k
      \quad \text{ for all } (\ell,k,\underline{i}) \in \mathcal{Q}.
    \end{split}
  \end{align}

  \textbf{Step~6 (contraction in $\boldsymbol{i}$).}
  For $i=0$ and $k=0$, we recall that $u_\ell^{0,0} = u_\ell^{0,\underline{i}} = u_\ell^{0,\star}$ by definition and
  hence $\textup{H}_\ell^{0,0} \eqreff{eq:equivalence}\simeq \textup{H}_\ell^0$. For $k \ge 1$, nested iteration
  $u_\ell^{k,0} = u_\ell^{k-1,\underline{i}}$, contraction of the exact Zarantonello
  iteration~\eqref{eq:ZarantonelloExact}, and energy equivalence~\eqref{eq:equivalence} imply that
  \begin{align*}
    \enorm{u_\ell^{k,\star} - u_\ell^{k,0}}
    &\le
    \enorm{u_\ell^\star - u_\ell^{k,\star}} + \enorm{u_\ell^\star - u_\ell^{k-1,\underline{i}}}
    \\
    &\eqreff*{eq:ZarantonelloExact}\le \,
    (q_{\textup{Zar}}^{\star}[\delta; 3M] + 1) \, \enorm{u_\ell^\star - u_\ell^{k-1,\underline{i}}}
    \eqreff*{eq:equivalence}\lesssim 2 \, \textup{H}_\ell^{k-1}.
  \end{align*}
  Therefore, by using the equivalence~\eqref{eq:equivalence} once more, we obtain that
  \begin{equation}\label{eq1:step9}
    \boxed{
      \textup{H}_\ell^{k,0} \lesssim \textup{H}_\ell^{(k-1)_+}
      \quad \text{for all } (\ell,k,0) \in \mathcal{Q},
      \quad \text{where } (k-1)_+ \coloneqq \max\{0, k-1 \}.
    }
  \end{equation}
  Let $(\ell, k, i) \in \mathcal{Q}$.
  It holds that
  \begin{align}\label{eq:algebra:1}
    \begin{split}
      \textup{H}_\ell^{k,i} \ &\eqreff*{eq:quasi-error}=\ \enorm{u_\ell^\star -
      u_\ell^{k,i}} + \enorm{u_\ell^{k,\star} - u_\ell^{k,i}} +
      \eta_\ell(u_\ell^{k,i})
      \\
      &\eqreff*{axiom:stability}\le \
      \textup{H}_\ell^{k,i-1} + (2 + C_{\textup{stab}}[10\tau]) \, \enorm{u_\ell^{k,i} - u_\ell^{k,i-1}}
      \\&
      \eqreff*{eq:algebraContraction}\le \
      \textup{H}_\ell^{k,i-1} + (2 + C_{\textup{stab}}[10\tau])(q_{\textup{alg}}+1) \,
      \enorm{u_\ell^{k,\star} - u_\ell^{k,i-1}}
      \eqreff*{eq:quasi-error}\lesssim \
      \textup{H}_\ell^{k,i-1},
    \end{split}
  \end{align}
  where $C_{\textup{stab}}[10 \tau]$ stems from the uniform bound~\eqref{eq:uniform:all} from
  Theorem~\ref{theorem:uniformBoundedness}. Hence, we obtain
  \begin{equation*}
    \textup{H}_\ell^{k,i} \lesssim \textup{H}_\ell^{k,i'} \simeq q_{\textup{alg}}^{i-i'} \,
    \textup{H}_\ell^{k, i'} \quad \text{ for all }
    (\ell, k, i) \in \mathcal{Q} \text{ with } 0 \le i' \le i \le i_{\textup{min}}.
  \end{equation*}

  For all $0 \le i' < i_{\textup{min}} \le i < \underline{i}[\ell,k]$, we obtain with an
  \textsl{a~posteriori} estimate based on
  the contraction of the Zarantonello iteration~\eqref{eq:ZarantonelloExact} (where
    $q_{\textup{Zar}}^{\star}=q_{\textup{Zar}}^{\star}[\delta, 2\tau]$ depends on $\tau$
  from~\eqref{crucial:mbounded}), the
  \textsl{a~posteriori} estimate~\eqref{eq:algebraic:estimate} for the algebraic solver, the failure of the
  accuracy criterion of~\eqref{eq:i_stopping_criterion}, and the contraction of the algebraic
  solver~\eqref{eq:algebraContraction} that
  \begin{align}\label{eq:algebra:3}
    \textup{H}_\ell^{k,i}
    \, &\eqreff*{eq:quasi-error}=\ \enorm{u_\ell^\star -
    u_\ell^{k,i}} + \enorm{u_\ell^{k,\star} - u_\ell^{k,i}} +
    \eta_\ell(u_\ell^{k,i}) \notag
    \\
    &\le \
    \enorm{u_\ell^\star - u_\ell^{k,\star}} + 2 \, \enorm{u_\ell^{k,\star} -
    u_\ell^{k,i}} + \eta_\ell(u_\ell^{k,i})
    \notag    \\& \mkern4mu
    \le \
    \frac{q_{\textup{Zar}}^{\star}[\delta; 2\tau]}{1-q_{\textup{Zar}}^{\star}[\delta; 2\tau]} \,
    \enorm{u_\ell^{k,i} - u_\ell^{k-1, \underline{i}}} \notag
    \\
    &
    \qquad + \Big(2 + \frac{q_{\textup{Zar}}^{\star}[\delta; 2\tau]}{1-q_{\textup{Zar}}^{\star}[\delta; 2\tau]} \Big) \,
    \enorm{u_\ell^{k,\star} - u_\ell^{k,i}} + \eta_\ell(u_\ell^{k,i})
    \notag  \\&\mkern4mu
    \eqreff*{eq:algebraic:estimate}\lesssim \
    \enorm{u_\ell^{k,i} - u_\ell^{k-1, \underline{i}}} + \enorm{u_\ell^{k,i} - u_\ell^{k, i-1}} +
    \eta_\ell(u_\ell^{k,i})
    \notag  \\&\mkern4mu
    \eqreff*{eq:i_stopping_criterion}\lesssim \
    \enorm{u_\ell^{k,i} - u_\ell^{k, i-1}}
    \eqreff{eq:algebraContraction}\lesssim
    \enorm{u_\ell^{k,\star} - u_\ell^{k, i-1}}
    \notag
    \\
    &\eqreff{eq:algebraContraction}\lesssim q_{\textup{alg}}^{i-i'} \,\enorm{u_\ell^{k,\star} - u_\ell^{k, i'}} \le
    q_{\textup{alg}}^{i-i'} \, \textup{H}_\ell^{k,i'},
  \end{align}
  Altogether, the combination of~\eqref{eq:algebra:1}--\eqref{eq:algebra:3} proves that
  \begin{equation}\label{eq3:step10}
    \boxed{
      \textup{H}_\ell^{k,i} \lesssim q_{\textup{alg}}^{i-i'} \, \textup{H}_\ell^{k,i'}
      \quad \text{ for all } (\ell,k,i) \in \mathcal{Q} \quad \text{ with } \quad 0 \le i' \le i \le
      \underline{i}[\ell,k],
    }
  \end{equation}
  where the hidden constant depends only on $q_{\textup{Zar}}^{\star}[\delta; 2\tau]$, $q_{\textup{alg}}$,
  $\lambda_{\textup{alg}}$,
  $C_{\textup{stab}}[10\tau]$, and $i_{\textup{min}}$.

  \textbf{Step~7 (tail summability with respect to $\boldsymbol{\ell}$, $\boldsymbol{k}$, and $\boldsymbol{i}$).}
  Finally, we observe that
  \begin{align*}
    \begin{split}
      \sum_{\substack{(\ell',k',i') \in \mathcal{Q} \\ |\ell',k',i'| > |\ell,k,i|}}  \textup{H}_{\ell'}^{k',i'}
      &=
      \sum_{i'=i+1}^{\underline{i}[\ell,k]} \textup{H}_{\ell}^{k,i'}
      + \sum_{k'=k+1}^{\underline{k}[\ell]} \sum_{i'=0}^{\underline{i}[\ell,k']} \textup{H}_{\ell}^{k',i'}
      + \sum_{\ell' = \ell+1}^{\underline{\ell}} \sum_{k'=0}^{\underline{k}[\ell']}
      \sum_{i'=0}^{\underline{i}[\ell',k']}
      \textup{H}_{\ell'}^{k',i'}
      \\&
      \eqreff*{eq3:step10}\lesssim \, \, \,
      \textup{H}_{\ell}^{k,i}
      + \sum_{k'=k+1}^{\underline{k}[\ell]} \textup{H}_\ell^{k',0}
      + \sum_{\ell' = \ell+1}^{\underline{\ell}} \sum_{k'=0}^{\underline{k}[\ell']} \textup{H}_{\ell'}^{k',0}
      \eqreff{eq1:step9}\lesssim
      \textup{H}_{\ell}^{k,i}
      + \! \! \!\sum_{\substack{(\ell',k',\underline{i}) \in \mathcal{Q} \\ |\ell',k',\underline{i}| >
      |\ell,k,\underline{i}|}}  \! \! \textup{H}_{\ell'}^{k'}
      \\&
      \eqreff*{eq:step8}\lesssim \, \, \,
      \textup{H}_\ell^{k,i}
      + \textup{H}_\ell^{k}
      \eqreff{eq2:proof}\lesssim
      \textup{H}_\ell^{k,i}
      + \textup{H}_\ell^{k,\underline{i}}
      \eqreff{eq3:step10}\lesssim
      \textup{H}_\ell^{k,i}
      \quad \text{ for all } (\ell,k,i) \in \mathcal{Q}.
    \end{split}
  \end{align*}
  Since $\mathcal{Q}$ is countable and linearly ordered,~\cite[Lemma~4.9]{axioms} applies and proves R-linear
  convergence~\eqref{eq:RLinearConvergence} of $\textup{H}_\ell^{k,i}$. This concludes the proof.
\end{proof}

Given full R-linear convergence from Theorem~\ref{theorem:RLinearConvergence}, then convergence rates with
respect to the degrees of freedom coincide with rates with respect to the overall computational cost, where
we recall $\mathtt{cost}(\ell,k,i)$ from~\eqref{eq:cost}. Since all essential arguments are provided, the
proof follows verbatim from~\cite[Corollary~16]{fps2023}.

\begin{corollary}[rates $\,\boldsymbol{\widehat{=}}\,$ complexity]\label{cor:rates=complexity}
  Suppose full R-linear convergence~\eqref{eq:RLinearConvergence}. Recall $\mathtt{cost}(\ell, k, i)$
  from~\eqref{eq:cost}. Then, for any $s > 0$, it holds that
  \begin{equation}\label{eq:single:complexity}
    M(s)
    \coloneqq
    \sup_{(\ell,k,i) \in \mathcal{Q}} (\#\mathcal{T}_\ell)^s \, \textup{H}_\ell^{k,i}
    \le \sup_{(\ell,k, i) \in \mathcal{Q}} \mathtt{cost}(\ell, k, i)^s \, \textup{H}_\ell^{k,i}
    \le C_{\rm cost} \, M(s),
  \end{equation}
  where the constant $C_{\rm cost} > 0$ depends only on $C_{\textup{lin}}$, $q_{\textup{lin}}$, and $s$.
  Moreover, there exists $s_0 > 0$ such that $M(s) < \infty$ for all $0 < s \le s_0$. \hfill \qed
\end{corollary}


\section{Optimal complexity}\label{section:optimality}

A formal approach to optimal complexity relies on the notion of approximation \linebreak
classes~\cite{bdd2004, stevenson2007, ckns2008, axioms}, which reads as follows: For
\(s > 0\), define
\begin{equation*}
  \| u^\star \|_{\mathbb{A}_s}
  \coloneqq
  \sup_{N \in \mathbb{N}_0}
  \bigl[
    ( N+1)^s \min_{\mathcal{T}_{\rm opt} \in \mathbb{T}_N }
    \eta_{\rm opt}(u^\star_{\rm opt})
  \bigr],
\end{equation*}
where $u_{\rm opt}^\star$ denotes the exact discrete solution associated with the optimal triangulation
$\mathcal{T}_{\rm opt} \in \mathbb{T}_N (\mathcal{T})$. For $s > 0$, we note that $\| u^\star
\|_{\mathbb{A}_s} < \infty$
means that the
sequence of estimators along optimally chosen meshes decreases at least as fast as $(N+1)^{-s}\simeq N^{-s}$.

Finally, we are in the position to present the third main result of this paper, namely optimal complexity of
Algorithm~\ref{algorithm:AILFEM}. Its proof relies, in essence, on perturbation arguments. More precisely,
sufficiently small \(\theta\) and \(\lambda_{\textup{lin}}\) are required to ensure that
Algorithm~\ref{algorithm:AILFEM} guarantees
convergence rate \(s\) with respect to the overall
computational cost (and time) if the solution \(u^\star\)
of~\eqref{eq:weakform} can be approximated at rate \(s\) in the sense of $\| u^\star \|_{\mathbb{A}_s} < \infty$.

\begin{theorem}[optimal complexity]\label{th:optimal_complexity}
  Define $\tau \coloneqq M + 3M \big(\frac{L[3M]}{\alpha}\big)^{1/2} \ge 4M$ with $M$ from~\eqref{eq:exact:bounded}.
  Let $0 < \delta < \min\{\frac{1}{L[5\tau]}, \frac{2\alpha}{L[2 \tau]^2}\}$
  to ensure validity of Theorem~\ref{theorem:uniformBoundedness}. Define
  \begin{equation}\label{eq:lambdastar}
    \lambda_{\textup{lin}}^\star \coloneqq \min\Big\{ 1, \Big(\frac{\alpha\,(1-q_{\mathcal{E}}^2)}{2\,q_{\mathcal{E}}^2}
    \Big)^{1/2}/C_{\textup{stab}}[3M]\Big\}.
  \end{equation}
  Suppose the axioms~\eqref{axiom:stability}--\eqref{axiom:discreteReliability}.
  {
    Let
    \(0 < \theta < 1\), $0 <\lambda_{\textup{alg}}$, and
    \(
      0 < \lambda_{\textup{lin}} < \lambda_{\textup{lin}}^\star
    \)
    such that
    \begin{equation}\label{eq:thetamark}
      0 <
      \theta_{\textup{mark}}
      \coloneqq
      \frac{(\theta^{1/2}+ \, \lambda_{\textup{lin}} / \lambda_{\textup{lin}}^\star)^{2}
      }{(1-\lambda_{\textup{lin}} / \lambda_{\textup{lin}}^\star)^{2}}
      <
      \theta^\star \coloneqq (1 + C_{\textup{stab}}[2M]^2 \, C_{\textup{rel}}^2)^{-1}
      <
      1.
    \end{equation}
  }
  Then, Algorithm~\ref{algorithm:AILFEM} guarantees, for all \(s > 0\), that
  \begin{equation}\label{eq:optimal_complexity}
    \sup_{
      \substack{
        (\ell,k, j) \in \mathcal{Q}
      }
    }
    \mathtt{cost}(\ell, k, i)^{s} \,
    \textup{H}_\ell^{k,j}
    \le
    C_{\textup{opt}} \,
    \max\{
      \| u^\star \|_{\mathbb{A}_s}, \,
      \textup{H}_{0}^{0,0}
    \}.
  \end{equation}
  The constant \(C_{\textup{opt}}>0\) depends only on $q_{\mathcal{E}}$, $\alpha$,
  \(C_{\textup{stab}}[10\tau]\), \(C_{\textup{rel}}\),
  \(C_{\textup{drel}}\), \(C_{\textup{mark}}\), \(C_{\textup{mesh}}\), $C_{\textup{lin}}$,
  $q_{\textup{lin}}$, \(\# \mathcal{T}_{0}\),
  and \(s\).
  In particular, there holds optimal complexity of
  Algorithm~\ref{algorithm:AILFEM}.
\end{theorem}

To prove the theorem, we require the following results on the estimator, which relies on sufficiently small
adaptivity parameter $\lambda_{\textup{lin}}>0$.
\begin{lemma}[{estimator equivalence}]\label{lemma:estimatorEquivalence}
  Suppose the assumptions of Theorem~\ref{theorem:uniformBoundedness}. Recall $\lambda_{\textup{lin}}^\star$
  from~\eqref{eq:lambdastar}.
  Then, for all $(\ell, \underline{k}, \underline{i}) \in \mathcal{Q}$ with $\underline{k}[\ell] < \infty$,
  it holds that
  \begin{subequations}\label{eq:corrigendum:equivalence}
    \begin{alignat}{2}
      &\eta_\ell(u_\ell^\star) &&\le
      ( 1 + \lambda_{\textup{lin}} / \lambda_{\textup{lin}}^\star ) \, \eta_\ell(u_\ell^{\underline{k},
      \underline{i}}),\label{eq:estimatorEquivalenceUpper}
      \intertext{and, for $0<\lambda_{\textup{lin}}< \lambda_{\textup{lin}}^\star$, we furthermore have that}
      ( 1 - \lambda_{\textup{lin}} / \lambda_{\textup{lin}}^\star ) \,\eta_\ell(u_\ell^{\underline{k},
      \underline{i}})\label{eq:estimatorEquivalenceLower}
      \le  \,& \eta_\ell(u_\ell^{\star})&&.
    \end{alignat}
  \end{subequations}
  For $0 < \lambda_{\textup{lin}} < \lambda_{\textup{lin}}^\star$, Dörfler marking for \(u_\ell^\star\)
  with parameter \(\theta_{\textup{mark}}\) from~\eqref{eq:thetamark} implies Dörfler marking for
  \(u_\ell^{\underline{k},
  \underline{i}}\) with parameter \(\theta\), i.e., for any
  \(\mathcal{R}_{\ell}
  \subseteq \mathcal{T}_{\ell}\), there holds the implication
  \begin{equation}\label{eq:equivalence_Doerfler}
    \theta_{\textup{mark}} \, \eta_{\ell}(u_\ell^\star)^2
    \le
    \eta_{\ell}(\mathcal{R}_{\ell}; u_\ell^\star)^2
    \quad
    \implies
    \quad
    \theta \, \eta_{\ell}(u_\ell^{\underline{k}, \underline{i}})^2
    \le
    \eta_{\ell}(\mathcal{R}_{\ell}; u_\ell^{\underline{k}, \underline{i}})^2.
  \end{equation}
\end{lemma}
\begin{proof}
  The proof consists of two steps.

  {\bf Step 1.}  First, we obtain from Remark~\ref{remark:stoppingAndMark}{\rm(ii)} that
  \begin{equation*}%
    \frac{\alpha}{2}\, \enorm{u_\ell^\star - u_\ell^{\underline{k}, \underline{i}}}^2
    \eqreff{eq:remark:linearization}\le
    \frac{\lambda_{\textup{lin}}^2\, q_{\mathcal{E}}^2}{1-q_{\mathcal{E}}^2} \,
    \eta_\ell(u_\ell^{\underline{k}, \underline{i}})^2.
  \end{equation*}
  Exploiting this together with stability~\eqref{axiom:stability}, nested
  iteration~\eqref{crucial:energyContraction}, and boundedness of the exact discrete
  solution~\eqref{eq:exact:bounded}, we obtain for any $\mathcal{U}_{\ell} \subseteq \mathcal{T}_\ell$ that
  \begin{align}
    \begin{split}\label{eq:estimatorEquivalenceStep2}
      \eta_\ell(\mathcal{U}_{\ell}; u_\ell^\star)
      \, &\eqreff*{axiom:stability}\le\, \,  \eta_\ell(\mathcal{U}_{\ell}; u_\ell^{\underline{k},
      \underline{i}})+ C_{\textup{stab}}[3M] \,
      \enorm{u^\star_\ell - u_\ell^{\underline{k}, \underline{i}}} \\
      & \le \, \, \eta_\ell(\mathcal{U}_{\ell}; u_\ell^{\underline{k}, \underline{i}}) +
      \lambda_{\textup{lin}}\,C_{\textup{stab}}[3M]\,\Big(\frac{2\,q_{\mathcal{E}}^2}{\alpha\,(1-q_{\mathcal{E}}^2)}
      \Big)^{1/2}\,
      \eta_\ell(u_\ell^{\underline{k}, \underline{i}}) \\
      &= \eta_\ell(\mathcal{U}_{\ell}; u_\ell^{\underline{k}, \underline{i}}) +
      \lambda_{\textup{lin}}/\lambda_{\textup{lin}}^\star\,
      \eta_\ell(u_\ell^{\underline{k}, \underline{i}}).
    \end{split}
  \end{align}
  The choice $\mathcal{U}_{\ell} = \mathcal{T}_{\ell}$ yields~\eqref{eq:estimatorEquivalenceUpper}.  The same
  arguments prove that
  \begin{equation}\label{eq:estimatorEquivalenceStep}
    \eta_\ell(\mathcal{U}_{\ell}; u_\ell^{\underline{k}, \underline{i}}) \le  \eta_\ell(\mathcal{U}_{\ell};
    u_\ell^{\star}) +
    \lambda_{\textup{lin}}/\lambda_{\textup{lin}}^\star\,\eta_\ell(u_\ell^{\underline{k}, \underline{i}}).
  \end{equation}
  For $0<\lambda_{\textup{lin}} < \lambda_{\textup{lin}}^\star$ and $\mathcal{U}_{\ell} = \mathcal{T}_{\ell}$, the
  rearrangement
  of~\eqref{eq:estimatorEquivalenceStep} proves~\eqref{eq:estimatorEquivalenceLower}.

  {\bf Step 2.} Let \(\mathcal{R}_{\ell} \subseteq \mathcal{T}_\ell\)
  satisfy
  \(\theta_{\textup{mark}}^{1/2} \, \eta_{\ell}(u_\ell^\star)
    \le
    \eta_{\ell}(\mathcal{R}_{\ell}; u_\ell^\star)
  \).
  Then,~\eqref{eq:estimatorEquivalenceStep2}--\eqref{eq:estimatorEquivalenceStep} prove
  \begin{align*}
    \bigl[1 - \lambda_{\textup{lin}} / \lambda_{\rm
    lin}^\star \bigr] \, \theta_{\textup{mark}}^{1/2} \,
    \eta_{\ell}(u_\ell^{\underline{k}, \underline{i}}) \,
    &\eqreff*{eq:estimatorEquivalenceLower}\le \,
    \theta_{\textup{mark}}^{1/2} \, \eta_{\ell}(u_\ell^\star)
    \le
    \eta_{\ell}(\mathcal{R}_{\ell}; u_\ell^\star)
    \\
    &\eqreff*{eq:estimatorEquivalenceStep2}\le \,
    \eta_{\ell}(\mathcal{R}_{\ell}; u_\ell^{\underline{k}, \underline{i}})
    +
    \lambda_{\textup{lin}} / \lambda_{\textup{lin}}^\star \,
    \eta_{\ell}(u_\ell^{\underline{k}, \underline{i}})
    \\
    &\eqreff*{eq:thetamark}= \,
    \eta_{\ell}(\mathcal{R}_{\ell}; u_\ell^{\underline{k}, \underline{i}})
    +
    \bigl[\theta_{\textup{mark}}^{1/2} \, \bigl(1\! -\! \lambda_{\textup{lin}} / \lambda_{\rm
    lin}^\star\bigr) - \theta^{1/2}\bigr] \,
    \eta_{\ell}(u_\ell^{\underline{k}, \underline{i}}).
  \end{align*}
  This yields
  \(
    \theta^{1/2} \, \eta_{\ell}(u_\ell^{\underline{k}, \underline{i}})
    \le
    \eta_{\ell}(\mathcal{R}_{\ell}; u_\ell^{\underline{k}, \underline{i}})
  \)
  and concludes the proof.
\end{proof}

\begin{proof}[{\bf \textit{Proof of Theorem~\ref{th:optimal_complexity}}}]
  By Corollary~\ref{cor:rates=complexity}, it is enough to show
  \begin{equation}\label{eq:proof_optimal_complexity}
    \sup \limits_{\substack{(\ell, k, i) \in \mathcal{Q}}}
    \bigl(\# \mathcal{T}_{\ell}\bigr)^{s} \, \textup{H}_{\ell}^{k, i}
    \lesssim
    \max\{
      \| u^\star \|_{\mathbb{A}_s},
      \textup{H}_{0}^{0,0}
    \}.
  \end{equation}
  Without loss of generality, we may suppose that
  \(\| u^\star \|_{\mathbb{A}_s} < \infty\). The proof is subdivided into two steps.

  \textbf{Step 1.}
  Let
  \(
    0
    <
    \theta_{\textup{mark}}
    \coloneqq
    (\theta^{1/2}+\lambda_{\textup{lin}} / \lambda_{\textup{lin}}^\star)^{2}  \, (1-\lambda_{\textup{lin}} /
    \lambda_{\textup{lin}}^\star)^{-{2}}
    <
    \theta^\star \coloneqq (1 + C_{\textup{stab}}[2M]^2 \, C_{\textup{rel}}^2)^{-1}
  \)
  and fix any $ 0 \le \ell' \le \underline{\ell} -1$. The validity of~\eqref{axiom:discreteReliability}
  and~\cite[Lemma~4.14]{axioms}
  guarantee the existence of a set \(\mathcal{R}_{\ell^\prime} \subseteq
  \mathcal{T}_{\ell^\prime}\) such
  that
  \begin{align}
    \# \mathcal{R}_{\ell^\prime}
    &\lesssim
    \| u^\star \|_{\mathbb{A}_s}^{1/s} \,
    [\eta_{\ell^\prime}(u_{\ell^\prime}^{\star})]^{-1/s}
    \label{eq:R_estimate},
    \\
    \theta_{\textup{mark}} \, \eta_{\ell^\prime}(u_{\ell^\prime}^{\star})
    &\le
    \eta_{\ell^\prime}(\mathcal{R}_{\ell^\prime},  u_{\ell^\prime}^\star),
    \notag
  \end{align}
  where the hidden constant depends only on \eqref{axiom:stability}--\eqref{axiom:discreteReliability}. By
  means of~\eqref{eq:equivalence_Doerfler} in Lemma~\ref{lemma:estimatorEquivalence}, we infer that
  $\mathcal{R}_{\ell'}$ satisfies the Dörfler marking~\eqref{eq:doerfler} in Algorithm~\ref{algorithm:AILFEM} with
  $\theta$, i.e., $  \theta \, \eta_{\ell'}(u_{\ell'}^{\underline{k}, \underline{i}})^2
  \le
  \eta_{\ell'}(\mathcal{R}_{\ell'}; u_{\ell'}^{\underline{k}, \underline{i}})^2$. Hence,
  since $0 < \theta <
  \theta_{\textup{mark}} < \theta^\star$, the optimality of Dörfler marking proves
  \begin{equation}\label{eq2:R_estimator_final}
    \# \mathcal{M}_{\ell'} \le C_{\textup{mark}} \, \# \mathcal{R}_{\ell'} \eqreff{eq:R_estimate}{\lesssim}
    \| u^\star \|_{\mathbb{A}_s}^{1/s} \,
    \bigl[
      \eta_{\ell^\prime}(u_{\ell^\prime}^{\star})
    \bigr]^{-1/s}.
  \end{equation}
  Moreover, full R-linear convergence~\eqref{eq:RLinearConvergence} together with \textsl{a~posteriori} error
  estimates for the final iterates~\eqref{eq:remark:linearization} and~\eqref{eq:algebraic:estimate},  the
  norm-energy equivalence~\eqref{eq:f4}, and estimator equivalence~\eqref{eq:corrigendum:equivalence} prove
  \begin{align}
    \begin{split}\label{eq:Delta_estimator}
      \textup{H}_{\ell'+1}^{0, \underline{i}}
      \, \,&   \eqreff*{eq:RLinearConvergence}\lesssim \,\,
      \textup{H}_{\ell'}^{\underline{k}, \underline{i}}
      \eqreff{eq:quasi-error}=   \enorm{u_{\ell'}^\star - u_{\ell'}^{\underline{k}, \underline{i}}} +
        \enorm{u_{\ell'}^{\underline{k},\star} - u_{\ell'}^{\underline{k},\underline{i}}}
      + \eta_{\ell'}(u_{\ell'}^{\underline{k},\underline{i}}) \\
      &\stackrel{\mathclap{\eqref{eq:algebraic:estimate},~\eqref{eq:f4}}}\lesssim \, \,  \, \, \, \,
        \enorm{u_{\ell'}^\star - u_{\ell'}^{\underline{k}, \underline{i}}} +
        [\mathcal{E}(u_{\ell'}^{\underline{k}, 0}) -
        \mathcal{E}(u_{\ell'}^{\underline{k},\underline{i}})]^{1/2}
      + \eta_{\ell'}(u_{\ell'}^{\underline{k},\underline{i}})\\
      &\stackrel{\mathclap{\eqref{eq:remark:linearization},~\eqref{eq:k_stopping_criterion}}}\lesssim \,
        \,  \, \, \, \,
        \eta_{\ell'}(u_{\ell'}^{\underline{k}, \underline{i}}) \eqreff{eq:corrigendum:equivalence}\lesssim
      \eta_{\ell'}(u_{\ell'}^\star).
    \end{split}
  \end{align}
  Consequently, a combination of~\eqref{eq2:R_estimator_final} and~\eqref{eq:Delta_estimator} concludes that
  \begin{equation}\label{eq:R_Delta}
    \# \mathcal{M}_{\ell'}
    \eqreff{eq2:R_estimator_final}\lesssim
    \| u^\star \|_{\mathbb{A}_s}^{1/s} \,
    \bigl[
      \eta_{\ell^\prime}(u_{\ell^\prime}^{\star})
    \bigr]^{-1/s}
    \eqreff{eq:Delta_estimator}\lesssim
    \| u^\star \|_{\mathbb{A}_s}^{1/s} \,
    \bigl[
    \textup{H}_{\ell^\prime+1}^{0, \underline{i}}\bigr]^{-1/s}.
  \end{equation}

  \textbf{Step 2.} For $(\ell, k, i) \in \mathcal{Q}$, full R-linear convergence \eqref{eq:RLinearConvergence} and
  the geometric series prove
  \begin{equation}\label{eq:lin_cv_sum}
    \! \! \! \! \! \sum_{
      \substack{
        (\ell^{\prime},k^{\prime},i^{\prime}) \in \mathcal{Q}
        \\
        |\ell^{\prime},k^{\prime},i^{\prime}|
        \le
        |\ell,k,i|
      }
    } \! \! \!
    (
      \textup{H}_{\ell^{\prime}}^{k^{\prime}, i^{\prime}}
    )^{-1/s}
    \eqreff{eq:RLinearConvergence}
    \lesssim
    (\textup{H}_{\ell}^{k, i})^{-1/s}
    \! \! \! \! \! \! \! \!
    \sum_{
      \substack{
        (\ell^{\prime},k^{\prime},i^{\prime}) \in \mathcal{Q}
        \\
        |\ell^{\prime},k^{\prime},i^{\prime}|
        \le
        |\ell,k,i|
      }
    }
    (q_{\textup{lin}}^{1/s})^{|\ell, k, i| - |\ell', k', i'|}
    \lesssim
    (\textup{H}_{\ell}^{k, i})^{-1/s}.
  \end{equation}
  We recall the mesh-closure estimate~\cite{bdd2004, stevenson2008, kpp2013, dgs2023}
  \begin{equation}\label{eq:meshclosure}
    \# \mathcal{T}_\ell - \# \mathcal{T}_0
    \le
    C_{\textup{mesh}} \sum_{\ell' = 0}^{\ell-1} \# \mathcal{M}_{\ell'} \quad \text{ for
    all } \ell \ge 0,
  \end{equation}
  where $C_{\textup{mesh}} > 1$ depends only on $\mathcal{T}_{0}$ and hence in particular on the dimension
  $d$.
  For $(\ell, k, i) \in \mathcal{Q}$, the preceding estimates show that
  \begin{align*}
    \# \mathcal{T}_{\ell} - \# \mathcal{T}_{0}
    \, \,&\eqreff{eq:meshclosure}
    \lesssim \, \,
    \sum_{\ell^\prime = 0}^{\ell-1}  \# \mathcal{M}_{\ell^\prime} \\
    &\eqreff{eq:R_Delta}
    \lesssim  \,
    \| u^\star \|_{\mathbb{A}_s}^{1/s} \,
    \sum_{\ell^{\prime} = 0}^{\ell-1}
    \bigl(
      \textup{H}_{\ell^{\prime}+1}^{0, \underline{i}}
    \bigr)^{-1/s}
    \\
    &\le
    \| u^\star \|_{\mathbb{A}_s}^{1/s} \! \!
    \mkern-12mu  \sum_{
      \substack{
        (\ell^{\prime},k^{\prime},i^{\prime}) \in \mathcal{Q}
        \\
        |\ell^{\prime},k^{\prime},i^{\prime}|
        \le
        |\ell,k,i|
      }
    }
    \mkern-12mu
    (\textup{H}_{\ell'}^{k', i'})^{-1/s}
    \eqreff{eq:lin_cv_sum}
    \lesssim \, \,
    \| u^\star \|_{\mathbb{A}_s}^{1/s} \,
    (\textup{H}_{\ell}^{k, i})^{-1/s}.
  \end{align*}
  Note that \(1 \le \# \mathcal{T}_{\ell} - \# \mathcal{T}_{0}\)
  yields
  $\# \mathcal{T}_{\ell} - \# \mathcal{T}_{0} +1 \le 2 \, (\#\mathcal{T}_{\ell} - \# \mathcal{T}_{0})$. Hence, we
  get that
  \begin{subequations}\label{eq:optimality}
    \begin{equation}
      (\# \mathcal{T}_{\ell} - \# \mathcal{T}_{0} + 1)^{s} \, \textup{H}_{\ell}^{k, i}
      \lesssim
      \| u^\star \|_{\mathbb{A}_s} \quad \text{ for all } (\ell, k, i) \in \mathcal{Q} \text{ with } \ell \ge 1.
    \end{equation}
    Theorem~\ref{theorem:RLinearConvergence} proves that
    \begin{equation}
      (\# \mathcal{T}_{\ell} - \# \mathcal{T}_{\ell} + 1)^{s} \,
      \textup{H}_{\ell}^{k, i}
      =
      \textup{H}_{0}^{k, i}
      \, \, \eqreff*{eq:RLinearConvergence}\lesssim \, \,
      \textup{H}_{0}^{0, 0} \quad \text{ for all } (\ell, k, i) \in \mathcal{Q} \text{ with } \ell =0.
    \end{equation}
  \end{subequations}
  For
  all
  $\mathcal{T}_{\ell} \in \mathbb{T}$, elementary calculation~\cite[Lemma~22]{bhp2017} shows that
  \begin{equation}\label{eq:bhp-lemma22}
    \# \mathcal{T}_{\ell} - \# \mathcal{T}_{0} +1
    \le
    \# \mathcal{T}_{\ell}
    \le
    \# \mathcal{T}_{0} \, (\# \mathcal{T}_{\ell} - \# \mathcal{T}_{0} +1).
  \end{equation}
  For all $(\ell, k, i) \in \mathcal{Q}$, we thus arrive at
  \begin{equation*}
    (\# \mathcal{T}_{\ell})^{s} \, \textup{H}_\ell^{k,i}
    \eqreff{eq:bhp-lemma22}
    \lesssim
    (\# \mathcal{T}_{\ell} - \# \mathcal{T}_{0} + 1)^{s} \, \textup{H}_\ell^{k,i}
    \eqreff{eq:optimality}
    \lesssim
    \max \{
      \| u^\star \|_{\mathbb{A}_s},
      \textup{H}_{0}^{0,0}
    \}.
  \end{equation*}
  This concludes the proof of~\eqref{eq:proof_optimal_complexity}.
\end{proof}


\section{Numerical experiments}\label{section:numerics}
The experiments are performed with the open-source software package MooAFEM~\cite{ip2022}. In the following,
Algorithm~\ref{algorithm:AILFEM} employs the optimal local \(hp\)-robust multigrid \linebreak
method~\cite{imps2022} as  algebraic solver. We remark that in our implementation the
condition~\eqref{eq:imin} is slightly relaxed to $|\mathcal{E}(u_\ell^{k, 0}) - \mathcal{E}(u_\ell^{k,i})| < 10^{-12}
\eqqcolon \mathtt{tol}$.

\begin{experiment}[modified sine-Gordon equation~{\cite[Experiment~5.1]{ahw2022}}]\label{example:gordon1}
  For $\Omega = (0,1)^2$, we consider
  \begin{equation}\label{eq:gordon1}
    -\Delta u^\star + (u^\star)^3+ \sin(u^\star) = f \quad \text{ in } \Omega \quad \text{ subject to }
    \quad u^\star= 0 \text{\ on } \partial \Omega
  \end{equation}
  with the monotone semilinearity $b(v) = v^{3}+ \sin(v)$, which fits into the locally Lipschitz continuous
  framework (cf.~\cite[Experiment~26]{bbimp2022cost}).  We choose $f$ such  that
  \begin{equation*}
    u^\star(x) = \sin(\pi x_1)\sin(\pi x_2).
  \end{equation*}
  For $T \in \mathcal{T}_H$, the refinement indicators $\eta_H(T; \cdot)$ read
  \begin{equation}\label{eq:estimator}
    \eta_H(T, v_H)^2 \coloneqq h_T^2 \,\| f + \Delta v_H  - b(v_H) \|_{L^2(T)}^2  +
    h_T \, \| \lbrack \! \lbrack\nabla v_H  \, \cdot \, \boldsymbol{n}  \rbrack \! \rbrack \|_{L^2(\partial T
    \cap \Omega)}^2.
  \end{equation}
  For $p=2$, damping parameter $\delta = 0.3$, and $i_{\textup{min}} = 1$, we stop the computation as soon as
  $\eta_\ell(u_\ell^{\underline{k}, \underline{i}})< 10^{-4}$. Table~\ref{table:parameterChoice} depicts the values
  of the weighted cost
  \begin{equation}\label{eq:weightedCost}
    \eta_\ell(u_\ell^{\underline{k}, \underline{i}}) \,\mathtt{cost}(\ell, \underline{k}, \underline{i})^{p/2}
  \end{equation}
  to determine the best parameter choice. For a fair comparison, the weighted cost from~\eqref{eq:weightedCost} balances the
    overachievement of the
  prescribed tolerance with the associated cumulative computational cost. We observe
  that the parameters $\theta \in \{0.3, 0.4\}$ and
  $\lambda_{\textup{lin}} \ge 0.5$ perform comparably well. The parameter $\lambda_{\textup{alg}}$ may be used for
  fine-tuning, but for moderate $\theta \in \{0.2, 0.3, 0.4, 0.5, 0.6\}$ and as soon as $\lambda_{\textup{lin}}$
  is set, the influence is comparably low.
  \begin{table}
    \centering
    \begin{tabular}{c}
      \includegraphics[width=0.9\textwidth]{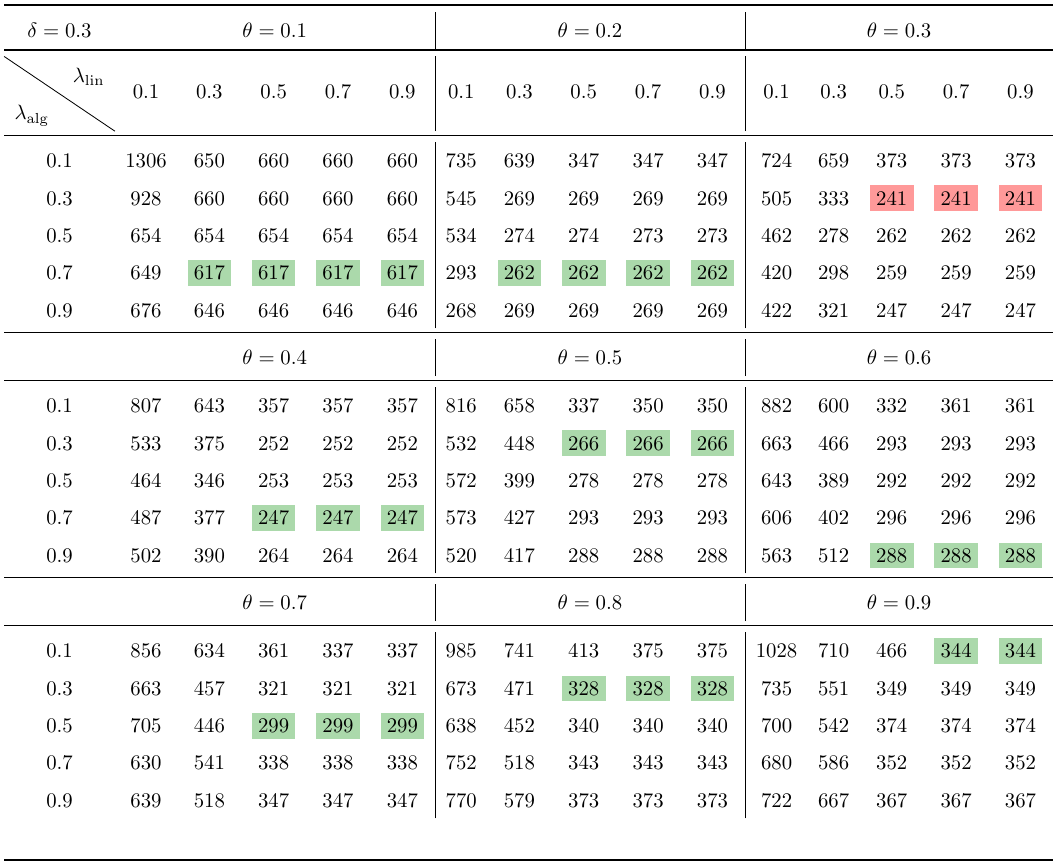}
    \end{tabular}
    \caption{The weighted cost~\eqref{eq:weightedCost} with \(p=2\) of the sine-Gordon
      problem~\eqref{eq:gordon1} for
      different adaptivity parameters $\lambda_{\textup{lin}}, \lambda_{\textup{alg}}, \theta \in \{0.1, 0.2, \ldots,
      0.9\}$ and fixed damping parameter $\delta =0.3$, where the mesh refinement is stopped if
      $\eta_\ell(u_\ell^{\underline{k}, \underline{i}}) < 10^{-4}$, where the $\theta$-blockwise minimal
      values are highlighted
    in \colorbox{pyGreen!40}{green} and the overall minimal value in~\colorbox{pyRed!40}{red}.}
    \label{table:parameterChoice}
  \end{table}

  For the following experiments, we set $\delta =0.3$, $\theta =0.3$, $\lambda_{\textup{lin}}=0.7$, and
  $\lambda_{\textup{alg}}=0.3$. Figure~\ref{fig:sineGordon1} depicts the error $\enorm{u^\star - u_\ell^{\underline{k},
  \underline{i}}}$ and the estimator $\eta_\ell(u_\ell^{\underline{k}, \underline{i}})$ over
  $\mathtt{cost}(\ell,\underline{k}, \underline{i})$
  (left) and over the cumulative time in seconds (right) for the displayed polynomial degrees $p \in \{1,2,3\}$.
  \begin{figure}
    \centering
    {\includegraphics[width=0.45\textwidth]{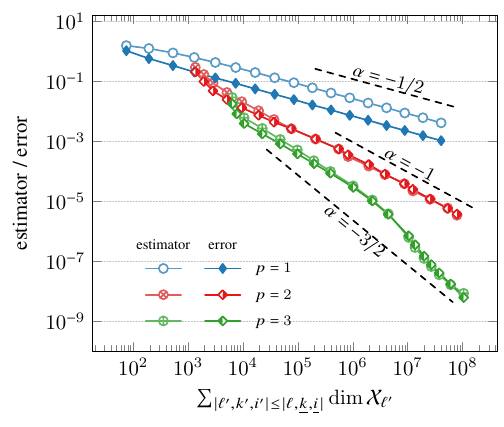}}\quad
    {\includegraphics[width=0.45\textwidth]{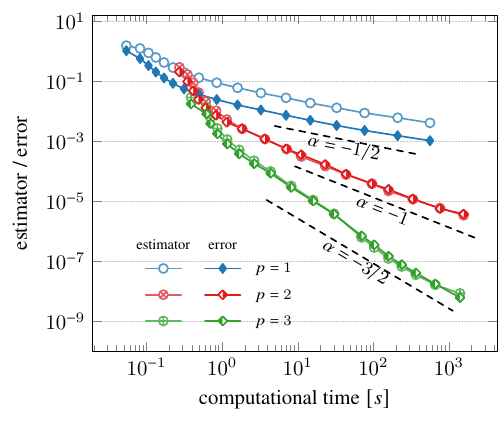}}
    \caption{Experiment~\ref{example:gordon1}: Convergence plots of the error
      $\vvvert u^\star-u_\ell^{\underline{k},\underline{i}} \vvvert$ (diamond) and the error estimator
      $\eta_\ell(u_\ell^{\underline{k}, \underline{i}})$ (circle) over $\mathtt{cost}(\ell, \underline{k}, \underline{i})$ (left) and over computational time in seconds (right).
      }\label{fig:sineGordon1}
  \end{figure}
  In both plots, the decay rate is of (expected) optimal order $p/2$ for $p \in \{1,2, 3\}$.
\end{experiment}

\begin{experiment}\label{experiment:quasilinear}
  We consider a globally Lipschitz continuous example  from~\cite[Section~5.3]{hpw2021} with Lipschitz
  constant $L = 2$ and monotonicity constant $\alpha = 1- 2 \exp(-\frac{3}{2})$ and hence $\delta =
  \alpha/L^2 \approx 0.138434919925785$ is a viable choice. For $d = 2$ and the L-shaped domain $(-1,1)^2
  \setminus \bigl([0,1] \times [-1,0]\bigr) \subset \mathbb{R}^2$, we seek $u^\star \in H_0^1(\Omega)$ such that
  \begin{equation*}
    - \operatorname{div}(\mu(|\nabla u^\star|^2) \nabla u^\star) = f \quad \text{ in } \Omega,
  \end{equation*}
  where $f$ is chosen such that $u^\star$ reads in polar coordinates $(r, \varphi) \in \mathbb{R}_{>0} \times
  [0, 2 \pi)$
  \[
    u^\star(r, \varphi) = r^{2/3} \, \sin\Big( \frac{2 \, \varphi}{3}\Big) \, (1-r \cos \varphi ) \, (1+r
    \cos \varphi) \, (1-r \sin \varphi ) \, (1+r \sin \varphi) \, \cos \varphi.
  \]
  This example has a singularity at the origin. We consider $p=1$, since stability~\eqref{axiom:stability} in
  the quasilinear case remains open for $p > 1$. Moreover, the parameters are $\theta =0.3$, $\lambda_{\textup{lin}}
  =0.7$, $\lambda_{\textup{alg}} =0.3$, and $i_{\textup{min}} = 1$.

  In Figure~\ref{fig:quasilinear}, we plot a sample solution (right) as well as convergence results of
  various error components (left) over the degrees of freedom. We observe that after a preasymptotic phase,
  optimal convergence rate $-1/2$ is restored for the exact error (diamond), the quasi-error
  $\textup{H}_\ell^{\underline{k},\underline{i}}$, the linearization error
  $\mathcal{E}(u_\ell^{\underline{k},0}) -
  \mathcal{E}(u_\ell^{\underline{k}, \underline{i}})$ (triangle), and the error estimator
  $\eta_\ell(u_\ell^{\underline{k},
  \underline{i}})$ (circle).
  \begin{figure}[h!]
    \centering
    \includegraphics[width=0.45\textwidth, valign=t]{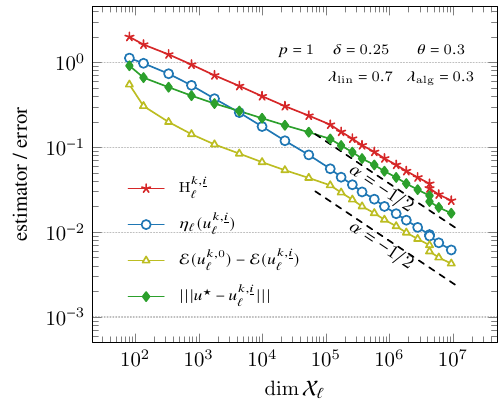}\quad
    \hfil    \includegraphics[width=0.45\textwidth,valign=t]{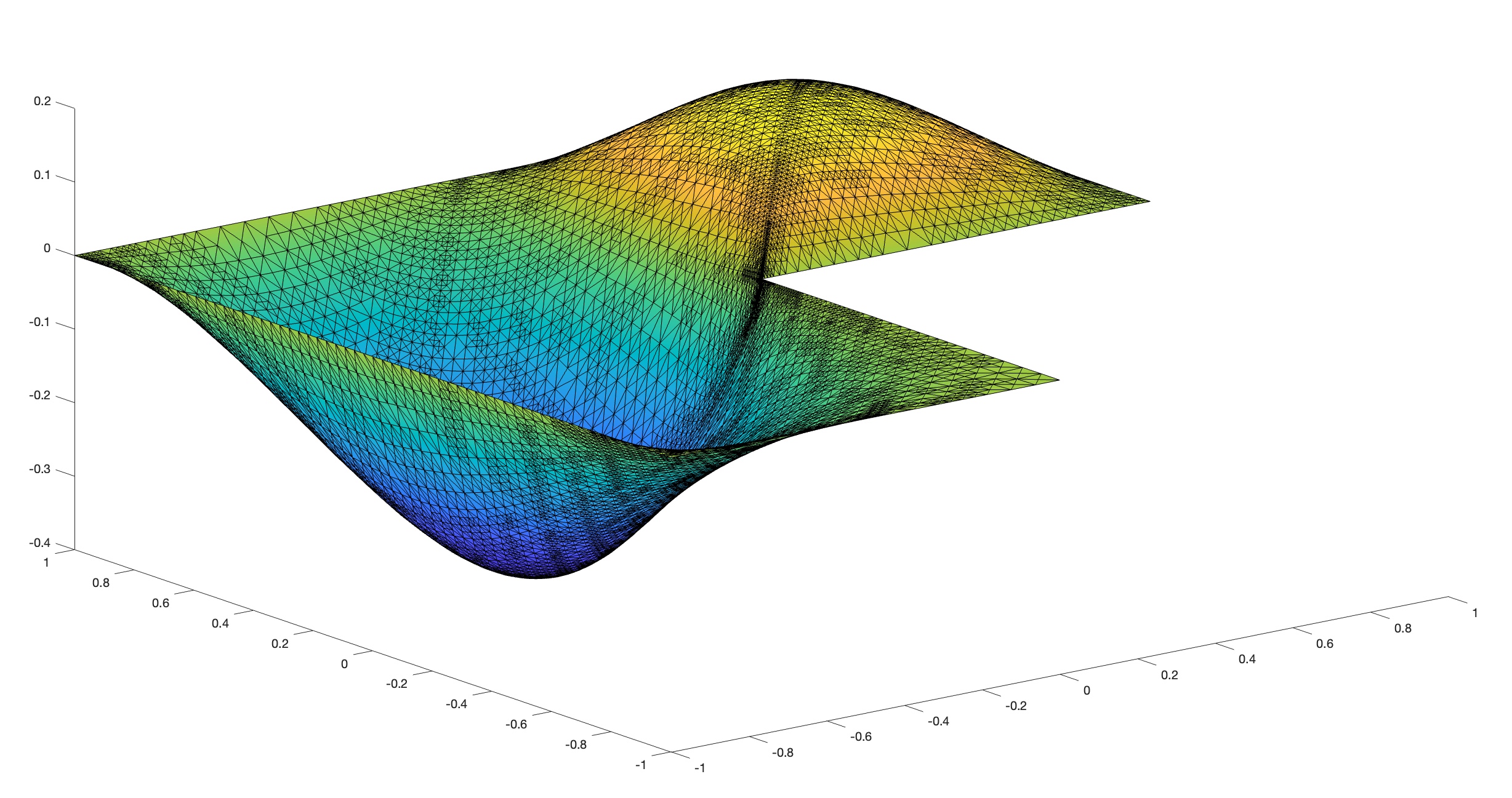}
    \caption{Experiment~\ref{experiment:quasilinear}: Convergence plots of various error components over the
      degrees of freedom (left). Right: Plot of the approximate solution $u_{13}^{1, 1}$ on $\mathcal{X}_{13}$ with $\#
    \mathcal{X}_{13} =10209$.}\label{fig:quasilinear}
  \end{figure}
\end{experiment}

\begin{experiment}[singularly perturbed sine-Gordon equation]\label{example:gordon2}
  This example is a variant of~{\cite[Experiment~5.2]{ahw2022}}. For $d=2$ and the L-shaped domain $(-1,1)^2
  \setminus \bigl([0,1] \times [-1,0]\bigr) \subset \mathbb{R}^2$, let $\varepsilon = 10^{-5}$ and consider
  \begin{equation*}
    -\varepsilon \Delta u^\star + u^\star+ (u^\star)^3+ \sin(u^\star) = 1 \quad \text{ in } \Omega \quad
    \text{ subject to } \quad u^\star= 0\text{\ on } \partial \Omega,
  \end{equation*}
  with the monotone semilinearity $b(v) = v^3+ \sin(v)$. In this case, the exact solution $u^\star$ is
  unknown.  We use the energy norm $\enorm{\, \cdot\, }^{2} = \varepsilon\, \dual{\nabla \cdot }{\nabla \cdot
  } + \dual{\cdot }{\cdot }$. The experiment is conducted with damping parameter $\delta =0.1$, $\lambda_{\rm
  alg} =0.7$, $\theta = 0.3$, and $i_{\textup{min}} = 1$. The refinement indicator~\eqref{eq:estimator} is modified
  along the lines of~\cite[Remark~4.14]{v2013} to
  \begin{equation*}
    \eta_H(T, v_H)^2 \coloneqq \hslash_T^2 \,\| f + \varepsilon \Delta v_H
    -v_H - b(v_H) \|_{L^2(T)}^2  + \hslash_T \, \| \lbrack \! \lbrack \varepsilon \, \nabla v_H  \,
    \cdot \, \boldsymbol{n} \rbrack \! \rbrack \|_{L^2(\partial T \cap \Omega)}^2,
  \end{equation*}
  where the scaling factors $\hslash_T = \min\{ \varepsilon^{-1/2} \, h_T, 1 \}$ ensure
  $\varepsilon$-robustness of the estimator.

  \begin{figure}
    \centering
    {\includegraphics[width=0.55\textwidth]{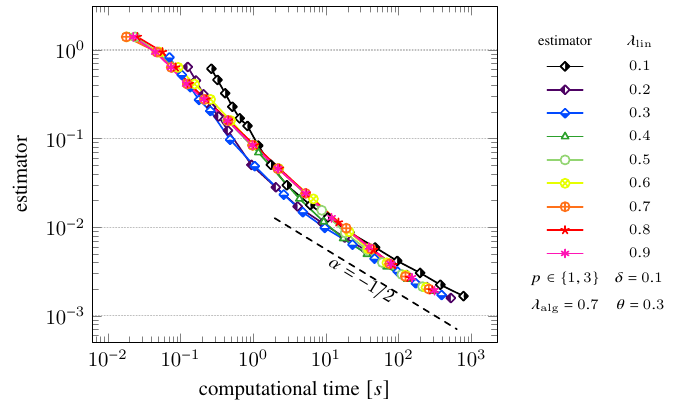}}
    {\includegraphics[width=0.4\textwidth]{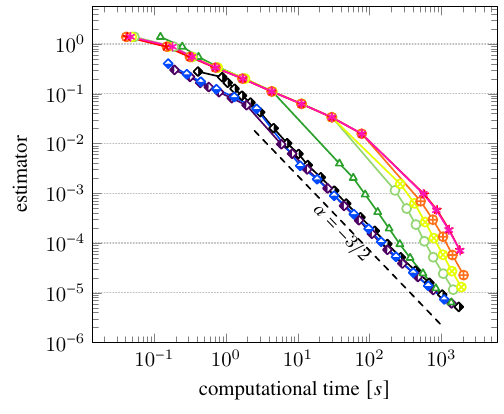}}
    \caption{Convergence plots of the error estimator $\eta_\ell(u^{\underline{k}, \underline{i}}_\ell)$ over
      computational
      time of Experiment~\ref{example:gordon2}. Left: Convergence plot for $p=1$ Right: Convergence plot for
    $p=3$.}\label{fig:singularSineGordon}
  \end{figure}
  In Figure~\ref{fig:singularSineGordon}, we plot the error estimator $\eta_\ell(u^{\underline{k},
  \underline{i}}_\ell)$ for
  all $(\ell, \underline{k},\underline{i}) \in \mathcal{Q}$ against the computational time for
  $\lambda_{\textup{lin}} \in
  \{0.1, 0.2,...,
  0.9\}$ and polynomial degrees $p \in \{1,3\}$. The decay rate is of (expected) optimal order $p/2$.  The
  choice of $\lambda_{\textup{lin}}$ does not play a major role in Figure~\ref{fig:singularSineGordon} (left) for
  $p=1$, but significantly prolongs the preasymptotic phase for $p=3$; see
  Figure~\ref{fig:singularSineGordon} (right). Figure~\ref{fig:singularSineGordonMeshes} shows meshes with
  $\#\mathtt{nDof} = 12475$ for $\lambda_{\textup{lin}} =0.2$ and $\# \mathtt{nDof} = 12152$ for $\lambda_{\rm
  lin} =0.7$. We see that the selection of a large $\lambda_{\textup{lin}}=0.7$ results in fewer
    linearization steps as well as fewer and algebraic solver steps but is subsequently taken care of by the mesh adaptivity. Hence,
  we observe stronger refinement in the interior.
  \begin{figure}
    \centering
    \includegraphics{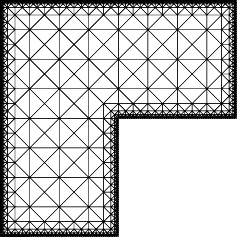}
    \hfil
    \includegraphics{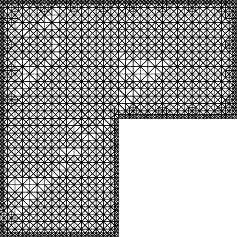}
    \caption{Mesh plot of Experiment~\ref{example:gordon2} for $p=3$. Left: Adaptivity parameter
      $\lambda_{\textup{lin}} = 0.2$. Right: Adaptivity parameter $\lambda_{\textup{lin}} =
    0.7$.}\label{fig:singularSineGordonMeshes}
  \end{figure}
  This experiment shows that Algorithm~\ref{algorithm:AILFEM} is suitable for a setting with dominating
  reaction given that a suitable norm on $\mathcal{X}$ is chosen. A large choice of $\lambda_{\textup{lin}}$
  seems possible, but
  pays off only after a long preasymptotic phase.

\end{experiment}



\printbibliography

\end{document}